\documentclass[11pt]{amsart}
\usepackage{amsmath,amssymb,amsthm,graphicx}
\usepackage[usenames]{color}
\usepackage[shortlabels]{enumitem}

\newtheorem{thm}{Theorem}
\newtheorem{lem}[thm]{Lemma}
\newtheorem{cor}[thm]{Corollary}

\theoremstyle{definition}
\newtheorem{defn}[thm]{Definition}

\newtheorem{rmk}[thm]{Remark}
\newtheorem{exmp}[thm]{Example}

\newcommand{\CPb}{\overline{\mathbb{CP}}{}^{2}}
\newcommand{\CP}{{\mathbb{CP}}{}^{2}}
\newcommand{\CPS}{{\mathbb{CP}}{}^{1}}

\newcommand{\Z}{\mathbb{Z}}

\newcommand{\red}{\textcolor{red}}

\newcommand{\blue}{\textcolor{blue}}

\title[Constructing Lefschetz fibrations via Daisy Substitutions]
{Constructing Lefschetz fibrations via Daisy Substitutions} 

\begin{document}

\author{Anar Akhmedov}
\address{School of Mathematics, 
University of Minnesota, 
Minneapolis, MN, 55455, USA}
\email{akhmedov@math.umn.edu}

\author{Naoyuki Monden}
\address{Department of Engineering Science, 
Osaka Electro-Communication University, 
Hatsu-cho 18-8, Neyagawa, 572-8530, Japan}
\email{monden@isc.osakac.ac.jp}




\begin{abstract} We construct new families of non-hyperelliptic Lefschetz fibrations by applying the daisy substitutions to the families of words $(c_1c_2 \cdots c_{2g-1}c_{2g}{c_{2g+1}}^2c_{2g}c_{2g-1} \cdots c_2c_1)^2 = 1$, $(c_1c_2 \cdots c_{2g}c_{2g+1})^{2g+2} = 1$, and $(c_1c_2 \cdots c_{2g-1}c_{2g})^{2(2g+1)} = 1$ in the mapping class group $\Gamma_{g}$ of the closed orientable surface of genus $g$, and study the sections of these Lefschetz fibrations. Furthemore, we show that the total spaces of some of these Lefschetz fibraions are irreducible exotic $4$-manifolds, and compute their Seiberg-Witten invariants. By applying the knot surgery to the family of Lefschetz fibrations obtained from the word $(c_1c_2 \cdots c_{2g}c_{2g+1})^{2g+2} = 1$ via daisy substitutions, we also construct an infinite family of pairwise non-diffeomorphic irreducible symplectic and non-symplectic $4$-manifolds homeomorphic to $(g^2 - g + 1)\CP\# (3g^{2} - g(k-3) + 2k + 3)\CPb$ for any $g \geq 3$, and $k = 2, \cdots, g+1$. 
\end{abstract}

\maketitle

\section{Introduction}

The Lefschetz fibrations are fundamental objects to study in $4$-dimensional topology. In the remarkable works \cite{D1, GS}, Simon Donaldson showed that every closed symplectic $4$-manifold admits a structure of Lefschetz pencil, which can be blown up at its base points to yield a Lefschetz fibration, and conversely, Robert Gompf showed that the total space of a genus $g$ Lefschetz fibration admits a symplectic structure, provided that the homology class of the fiber is nontrivial. Given a Lefschetz fibration over $\mathbb{S}^2$, one can associate to it a word in the mapping class group of the fiber composed solely of right-handed Dehn twists, and conversely, given such a factorization in the mapping class group, one can construct a Lefschetz fibration over $\mathbb{S}^2$ (see for example \cite{GS}). 

Recently there has been much interest in trying to understand the topological interpretation of various relations in the mapping class group. A particularly well understood case is the daisy relation, which corresponds to the symplectic operation of rational blowdown \cite{EG, EMVHM}. Another interesting problem, which is still open, is whether any Lefschetz fibration over $\mathbb{S}^2$ admits a section (see for example \cite{Sm}). Furthermore, one would like to determine how many disjoint sections the given Lefschetz fibration admits. The later problem has been studied for the standard family of hyperelliptic Lefschetz fibrations (with total spaces $\CP\#(4g+5)\CPb$) in \cite{ko, Sinem, Tan}, using the computations in the mapping class group, and such results are useful in constructing (exotic) Stein fillings \cite{ao2, AO}.       

Motivated by these results and problems, our goal in this paper is to construct new families of Lefschetz fibrations over $\mathbb{S}^{2}$ by applying the sequence of daisy substitutions and conjugations to the hyperelliptic words $(c_1c_2 \cdots c_{2g-1}c_{2g}{c_{2g+1}}^2c_{2g}c_{2g-1} \cdots c_2c_1)^2 = 1$, $(c_1c_2 \cdots c_{2g}c_{2g+1})^{2g+2} = 1$, and $(c_1c_2 \cdots c_{2g-1}c_{2g})^{2(2g+1)} = 1$ in the mapping class group of the closed orientable surface of genus $g$ for any $g \geq 3$ and study the sections of these Lefschetz fibrations (cf. Theorems \ref{4.1}-\ref{4.4}). Furthermore, we show that the total spaces of our Lefschetz fibraions given by the last two words are irreducible exotic symplectic $4$-manifolds, and compute their Seiberg-Witten invariants (cf. Theorem \ref{theorem1}). The analogues (but weaker) results for special case of $g = 2$, using the lantern substitutions only, were obtained in \cite{EG, AP}. We would like to remark that the mapping class group computations in our paper are more involved and subtle than in \cite{EG, AP}. One family of examples, obtained from the fiber sums of the Lefschetz fibrations using daisy relations, were studied in \cite{EMVHM}. However, the examples obtained in \cite{EMVHM} have larger topology, and computations of Seiberg-Witten invariants, and study of sections were not addressed in \cite{EMVHM}. Moreover, we prove non-hyperellipticity of our Lefschetz fibrations and provide some criterias for non-hyperellipticity under the daisy substitutions (cf. Theorem \ref{nonhyperelliptic}). Some of our examples can be used to produce the families of non-isomorphic Lefschetz fibrations over $\mathbb{S}^{2}$ with the same total spaces and exotic Stein fillings. We hope to return these examples in future work.        

The organization of our paper is as follows. In Sections 2 and 3 we recall the main definitions and results that will be used throughout the paper. In Section 4, we prove some technical lemmas, important in the proofs of our main theorems. In Sections 5 and 6, we construct new families of Lefschetz fibrations by applying the daisy substitutions to the words given above, study the sections and prove non-hyperellipticity of our Lefschetz fibrations (Theorems \ref{4.1}, \ref{4.2}, \ref{4.7}, \ref{4.3}, \ref{4.4}, \ref{nonhyperelliptic}). Finally, in Section 7, we prove that the total spaces of some of these Lefschetz fibrations are exotic symplectic $4$-manifolds, which we veirfy by computing their Seiberg-Witten invariants and obtain infinite family of exotic $4$-manifolds via knot surgery (Theorems \ref{theorem1}, \ref{theorem3}), and make some remarks and raise questions. We would like to remark that the main technical content of our paper is more algebraic since our proofs rely heavily on mapping class group techniques. It is possible to pursue a more geometric approach (see Example~\ref{Ex}), but such approach alone does not yield the optimal results as presented here.

\section{Mapping Class Groups} Let $\Sigma_{g}^n$ be a $2$-dimensional, compact, oriented, and connected surface of genus $g$ with $n$ boundary components. 
Let $Diff^{+}\left( \Sigma_{g}^n\right)$ be the group of all orientation-preserving self-diffeomorphisms of $\Sigma_{g}^n$ which are the identity on the boundary and $ Diff_{0}^{+}\left(\Sigma_{g}\right)$ be the subgroup of $Diff^{+}\left(\Sigma_{g}\right)$ consisting of all orientation-preserving self-diffeomorphisms that are isotopic to the identity. 
The isotopies are also assumed to fix the points on the boundary. 
\emph{The mapping class group} $\Gamma_{g}^n$ of $\Sigma_{g}^n$ is defined to be the group of isotopy classes of orientation-preserving diffeomorphisms of $\Sigma_{g}^n$, i.e.,
\[
\Gamma_{g}^n=Diff^{+}\left( \Sigma_{g}^n\right) /Diff_{0}^{+}\left(
\Sigma_{g}^n\right) .
\]
For simplicity, we write $\Sigma_g = \Sigma_g^0$ and $\Gamma_g = \Gamma_g^0$. 

The hyperelliptic mapping class group $H_{g}$ of $\Sigma_{g}$ is defined as the subgroup of $\Gamma_g$ consisting of all isotopy classes commuting with the isotopy class of the hyperelliptic involution $\iota: \Sigma_{g}\rightarrow \Sigma_{g}$. 

\begin{defn} Let $\alpha$ be a simple closed curve on $\Sigma_{g}^n$. A \emph{right handed} (or positive) \emph{Dehn twist} about $\alpha$ is a diffeomorphism of $t_{\alpha}: \Sigma_{g}^n\rightarrow \Sigma_{g}^n$ obtained by cutting the surface $\Sigma_{g}^n$ along $\alpha$ and gluing the ends back after rotating one of the ends $2\pi$ to the right. 
\end{defn}

It is well-known that the mapping class group $\Gamma_g^n$ is generated by Dehn twists. It is an elementary fact that the conjugate of a Dehn twist is again a Dehn twist: if $\phi: \Sigma_{g}^n\rightarrow \Sigma_{g}^n$ is an orientation-preserving diffeomorphism, then $\phi \circ t_\alpha \circ \phi^{-1} = t_{\phi(\alpha)}$. The following lemma is easy to verify (see \cite{I} for a proof).

\begin{lem} \label{com&braid.lem} Let $\alpha$ and $\beta$ be two simple closed curves on $\Sigma_{g}^n$. If $\alpha$ and $\beta$ are disjoint, then their corresponding Dehn twists satisfy the commutativity relation: $t_{\alpha}t_{\beta}=t_{\beta}t_{\alpha}.$ If $\alpha$ and $\beta$ transversely intersect at a single point, then their corresponding Dehn twists satisfy the braid relation: $t_{\alpha}t_{\beta}t_{\alpha}=t_{\beta}t_{\alpha}t_{\beta}.$
\end{lem}

\subsection{Daisy relation and daisy substitution}
We recall the definition of the daisy relation (see \cite{PVHM}, \cite{EMVHM}, \cite{BKM}).

\begin{defn}\label{daisy}\rm
Let $\Sigma_0^{p+2}$ denote a sphere with $p+2$ boundary components $(p\geq 2)$. Let $\delta_0, \delta_1, \delta_2,\ldots, \delta_{p+1}$ be the $p$ boundary curves of $\Sigma_0^{p+2}$ and let $x_1, x_2,\ldots, x_{p+1}$ be the interior curves as shown in Figure~\ref{daisy}. Then, we have the \textit{daisy relation of type $p$}:
\begin{align*}
t_{\delta_0}^{p-1}t_{\delta_1}t_{\delta_2}\cdots t_{\delta_{p+1}}=t_{x_1}t_{x_2}\cdots t_{x_{p+1}}.
\end{align*}
We call the following relator the \textit{daisy relator of type $p$}: 
\begin{align*}
t_{\delta_{p+1}}^{-1}\cdots t_{\delta_2}^{-1}t_{\delta_1}^{-1}t_{\delta_0}^{-p+1}t_{x_1}t_{x_2}\cdots t_{x_{p+1}} \ (=1).
\end{align*}
\end{defn}

\begin{rmk}\rm
When $p=2$, the daisy relation is commonly known as the \textit{lantern relation} (see \cite{De}, \cite{Jo}). 
\end{rmk}

\begin{figure}[ht]
\begin{center}
\includegraphics[scale=.43]{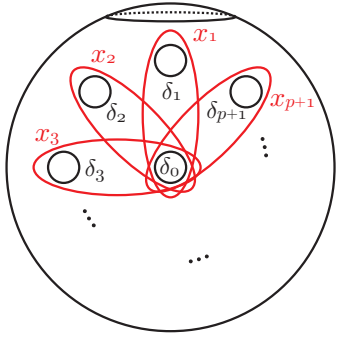}
\caption{Daisy relation}
\label{fig:hyper}
\end{center}
\end{figure}

We next introduce a daisy substitution, a substitution technique introduced by T. Fuller. 

\begin{defn}\label{daisy substitution}\rm
Let $d_1,\ldots,d_m$ and $e_1,\ldots, e_n$ be simple closed curves on $\Gamma_g^n$, and let $R$ be a product $R=t_{d_1}t_{d_2}\cdots t_{d_l}t_{e_m}^{-1}\cdots t_{e_2}^{-1}t_{e_1}^{-1}$. Suppose that $R=1$ in $\Gamma_g^n$. Let $\varrho$ be a word in $\Gamma_g^n$ including $t_{d_1}t_{d_2}\cdots t_{d_l}$ as a subword: 
\begin{align*}
\varrho=U\cdot t_{d_1}t_{d_2}\cdots t_{d_l} \cdot V, 
\end{align*}
where $U$ and $V$ are words. Thus, we obtain a new word in $\Gamma_g^n$, denoted by $\varrho^\prime$, as follows:
\begin{align*}
\varrho^\prime:&=U\cdot t_{e_1}t_{e_2}\cdots t_{e_m} \cdot V.
\end{align*}
Then, we say that $\varrho^{\prime}$ is obtained by applying a $R$-\textit{substitution} to $\varrho$. In particular, if $R$ is a daisy relator of type $p$, then we say that $\varrho^{\prime}$ is obtained by applying a \textit{daisy substitution of type $p$} to $\varrho$. 
\end{defn}

\section{Lefschetz fibrations}
\begin{defn}\label{LF}\rm
Let $X$ be a closed, oriented smooth $4$-manifold. A smooth map $f : X \rightarrow \mathbb{S}^2$ is a genus-$g$ \textit{Lefschetz fibration} if it satisfies the following condition: \\
(i) $f$ has finitely many critical values $b_1,\ldots,b_m \in S^2$, and $f$ is a smooth $\Sigma_g$-bundle over $\mathbb{S}^2-\{b_1,\ldots,b_m\}$, \\
(ii) for each $i$ $(i=1,\ldots,m)$, there exists a unique critical point $p_i$ in the \textit{singular fiber} $f^{-1}(b_i)$ such that about each $p_i$ and $b_i$ there are local complex coordinate charts agreeing with the orientations of $X$ and $\mathbb{S}^2$ on which $f$ is of the form $f(z_{1},z_{2})=z_{1}^{2}+z_{2}^{2}$, \\
(i\hspace{-.1em}i\hspace{-.1em}i) $f$ is relatively minimal (i.e. no fiber contains a $(-1)$-sphere.)
\end{defn}

Each singular fiber is obtained by collapsing a simple closed curve (the \textit{vanishing cycle}) in the regular fiber. The monodromy of the fibration around a singular fiber is given by a right handed Dehn twist along the corresponding vanishing cycle. 
For a genus-$g$ Lefschetz fibration over $\mathbb{S}^2$, the product of right handed Dehn twists $t_{v_i}$ about the vanishing cycles $v_i$, for $i = 1,\ldots, m$, gives us the global monodromy of the Lefschetz fibration, the relation $t_{v_1} t_{v_2} \cdots t_{v_m}=1$ in $\Gamma_g$. This relation is called the \textit{positive relator}. Conversely, such a positive relator defines a genus-$g$ Lefschetz fibration over $\mathbb{S}^2$ with the vanishing cycles $v_1,\ldots, v_m$.

According to theorems of Kas \cite{Kas} and Matsumoto \cite{Ma}, if $g\geq 2$, then the isomorphism class of a Lefschetz fibration is determined by a positive relator modulo simultaneous conjugations 
\begin{align*}
t_{v_1} t_{v_2} \cdots  t_{v_m} \sim t_{\phi(v_1)} t_{\phi(v_2)} \cdots t_{\phi(v_m)} \ \ {\rm for \ all} \ \phi \in M_g
\end{align*}
and elementary transformations 
\begin{align*}
&t_{v_1} \cdots t_{v_{i-1}} t_{v_i} t_{v_{i+1}} t_{v_{i+2}} \cdots t_{v_m}& &\sim& &t_{v_1} \cdots t_{v_{i-1}} t_{t_{v_i}(v_{i+1})} t_{v_i} t_{v_{i+2}} \cdots t_{v_m},&\\
&t_{v_1} \cdots t_{v_{i-2}} t_{v_{i-1}} t_{v_i} t_{v_{i+1}} \cdots t_{v_m}& &\sim& &t_{v_1} \cdots t_{v_{i-2}} t_{v_i} t_{t_{v_i}^{-1}(v_{i-1})} t_{v_{i+1}} \cdots  t_{v_m}.&
\end{align*}
Note that $\phi t_{v_i}\phi^{-1}=t_{\phi(v_i)}$. 
We denote a Lefschetz fibration associated to a positive relator $\varrho \ \in \Gamma_g$ by $f_\varrho$. 

For a Lefschetz fibration $f:X\rightarrow \mathbb{S}^2$, a map $\sigma:\mathbb{S}^2\rightarrow X$ is called a \textit{section} of $f$ if $f\circ \sigma={\rm id}_{\mathbb{S}^2}$. 
We define the self-intersection of $\sigma$ to be the self-intersection number of the homology class $[\sigma(\mathbb{S}^2)]$ in $H_2(X;\Z)$. 
Let $\delta_1,\delta_2,\ldots,\delta_n$ be $n$ boundary curves of $\Sigma_g^n$. 
If there exists a lift of a positive relator $\varrho = t_{v_1} t_{v_2} \cdots t_{v_m} = 1$ in $\Gamma_g$ to $\Gamma_g^n$ as 
\begin{align*}
t_{\tilde{v}_1} t_{\tilde{v}_2} \cdots t_{\tilde{v}_m} = t_{\delta_1} t_{\delta_2} \cdots t_{\delta_n}, 
\end{align*}
then $f_\varrho$ admits $n$ disjoint sections of self-intersection $-1$. 
Here, $t_{\tilde{v}_i}$ is a Dehn twist mapped to $t_{v_i}$ under $\Gamma_g^n \to \Gamma_g$. 
Conversely, if a genus-$g$ Lefschetz fibration admits $n$ disjoint sections of self-intersection $-1$, then we obtain such a relation in $\Gamma_g^n$. 

Next, let us recall the signature formula for hyperelliptic Lefschetz fibrations, which is due to Matsumoto and Endo. We will make use of this formula in Section \ref{new words}, where we prove that all our Lefschetz fibrations obtained via daisy substitutions are non-hyperelliptic.  

\begin{thm}[\cite{Ma1},\cite{Ma},\cite{E}]\label{sign} Let $f:X\rightarrow \mathbb{S}^2$ be a genus $g$ hyperelliptic Lefschetz fibration. Let $s_0$ and $s=\Sigma_{h=1}^{[g/2]}s_h$ be the number of non-separating and separating vanishing cycles of $f$, where $s_h$ denotes the number of separating vanishing cycles which separate the surface of genus $g$ into two surfaces, one of which has genus $h$. Then, we have the following formula for the signature 
\begin{eqnarray*}
\sigma(X)=-\frac{g+1}{2g+1}s_0+\sum_{h=1}^{[\frac{g}{2}]}\left(\frac{4h(g-h)}{2g+1}-1\right)s_{h}.
\end{eqnarray*}
\end{thm}

\subsection{Spinness criteria for Lefschetz fibrations}

In this subsection, we recall two Theorems, due to A. Stipsicz (\cite{St}), concerning the non-spinness and spinness of the Lefschetz fibrations over $\mathbb{D}^2$ and $\mathbb{S}^2$. We will use them to verify our familes of Lefschetz fibrations in Theorem \ref{theorem1} all are non-spin. Since Rohlin's Theorem can not be used to verify non-spinness when the signature of our Lefchetz fibrations is divisible by $16$, Stipsicz's results will be more suitable for our purpose.   

Let $f:X \rightarrow \mathbb{D}^2$ be a Lefschetz fibration over disk, and $F$ denote the generic fiber of $f$. Denote the homology classes of the vanishing cycles of the given fibration by $v_{1}, \cdots, v_{m} \in H_{1}(F; \mathbb{Z}_{2})$. 
        
\begin{thm}\label{Spin1} \cite{St}. The Lefschetz fibration $f:X \rightarrow \mathbb{D}^2$ is not spin if and only if there are $l$ vanishing cylces $v_{1}$, $\cdots$, $v_{l}$ such that $v = \sum_{i=1}^{l} v_{i}$ is also a vanishing cycle, and $l +  \sum_{1 \leq i < j \leq l} v_{i} \cdot v_{j} \equiv 0 (\textrm{mod}\ 2)$.       
\end{thm}

Note that the above theorem imples that if the Lefschetz fibration has the separating vanishing cycle then its total space is not spin. To see this, set $l = 0$ and take the empty sum to be $0$.    

\begin{thm}\label{Spin2} \cite{St}. The Lefschetz fibration $f:X \rightarrow \mathbb{S}^2$ is spin if and only if $X \setminus \nu(F)$ is spin and for some dual $\sigma$ of $F$ we have $\sigma^2 \equiv 0 (\textrm{mod}\ 2)$. \end{thm}

\subsection{Three familes of hyperelliptic Lefschetz fibrations}

In this subsection, we introduce three well-known familes of hyperelliptic Lefschetz fibrations, which will serve as building blocks in our construction of new Lefschetz fibrations. 
Let $c_1$, $c_2$, .... , $c_{2g}$, $c_{2g+1}$ denote the collection of simple closed curves given in Figure~\ref{fig:hyper}, and  $c_{i}$ denote the right handed Dehn twists $t_{c_i}$ along the curve $c_i$. It is well-known that the following relations hold in the mapping class group $\Gamma_g$:  

\begin{equation}
\begin{array}{l}
H(g) = (c_1c_2 \cdots c_{2g-1}c_{2g}{c_{2g+1}}^2c_{2g}c_{2g-1} \cdots c_2c_1)^2 = 1,  \\
I(g) = (c_1c_2 \cdots c_{2g}c_{2g+1})^{2g+2} = 1,  \\ 
G(g) = (c_1c_2 \cdots c_{2g-1}c_{2g})^{2(2g+1)} = 1. 
\end{array}
\end{equation}

\begin{figure}[ht]
\begin{center}
\includegraphics[scale=.40]{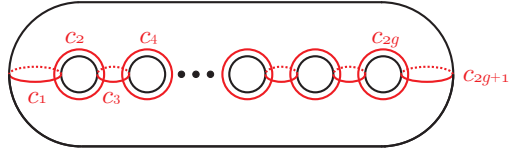}
\caption{Vanishing Cycles of the Genus $g$ Lefschetz Fibration on $X(g)$, $Y(g)$, and $Z(g)$}
\label{fig:hyper}
\end{center}
\end{figure}

Let $X(g)$, $Y(g)$ and $Z(g)$ denote the total spaces of the above genus $g$ hyperelliptic Lefschetz fibrations given by the monodromies $H(g) = 1$, $I(g) = 1$, and $J(g) = 1$ respectively, in the mapping class group $\Gamma_g$. For the first monodromy relation, the corresponding genus $g$ Lefschetz fibrations over $\mathbb{S}^2$ has total space $X(g) = \CP\#(4g+5)\CPb$, the complex projective plane blown up at $4g+5$ points. In the case of second and third relations, the total spaces of the corresponding genus $g$ Lefschetz fibrations over $\mathbb{S}^2$ are also well-known families of complex surfaces. For example, $Y(2) = K3\#2\CPb$ and $Z(2)$ = \emph{Horikawa surface}, respectively. In what follows, we recall the branched-cover description of the $4$-manifolds $Y(g)$ and $Z(g)$, which we will use in the proofs of our main results. The branched-cover description of $X(g)$ is well-known and we refer the reader to (\cite{GS}, Remark 7.3.5, p.257).      

\begin{lem}\label{E} The genus $g$ Lefschetz fibration on $Y(g)$ over $\mathbb{S}^2$ with the monodromy $(c_1c_2 \cdots c_{2g+1})^{2g+2} = 1$ can be obtained as the double branched covering of $\CP\#\CPb$ branched along a smooth algebraic curve $B$ in the linear system $|2(g+1)\tilde{L}|$, where $\tilde{L}$ is the proper transform of line $L$ in $\CP$ avoiding the blown-up point. Furthermore, this Lefschetz fibration admits two disjoint $-1$ sphere sections.   
\end{lem}

\begin{proof} We will follow the proof of Lemma 3.1 in \cite{Ar1}, where $g=2$ case have been considered (see also the discussion in \cite{AK}), and make necessary adjustments where needed. Let $D$ denote an algebraic curve of degree $d$ in $\CP$. We fix a generic projection map $\pi : \CP \setminus {pt} \rightarrow \mathbb{CP}^1$ such that the pole of $\pi$ does not belong to $D$. It was shown in \cite{MT} that the braid monodromy of $D$ in $\CP$ is given via a braid factorization. More specifically, the braid monodromy around the point at infinity in $\mathbb{CP}^1$, which is given by the central element $\Delta^2$ in $B_{d}$, can be written as the product of the monodromies about the critical points of $\pi$. Hence, the factorization $\Delta^2 = (\sigma_{1} \cdots \sigma_{d-1})^{d}$ holds in the braid group $B_{d}$, where $\sigma_{i}$ denotes a positive half-twist exchanging two points, and fixing the remaining $d-2$ points.

Now let us degenerate the smooth algebraic curve $B$ in $\CP\#\CPb$ into a union of $2(g+1)$ lines in a general position. By the discussion above, the braid group factorization corresponding to the configuration $B$ is given by $\Delta^2 = (\sigma_{1} \sigma_{2} \cdots \sigma_{2g} \sigma_{2g+1})^{2g+2}$. Now, by lifting this braid factorization to the mapping class group of the genus $g$ surface, we obtain that the monodromy factorization  $(c_1c_2 \cdots c_{2g+1})^{2g+2} = 1$ for the corresponding double branched covering.

Moreover, observe that a regular fiber of the given fibration is a two fold cover of a sphere in $\CP\#\CPb$ with homology class $f = h - e_{1}$ branched over $2(g+1)$ points, where $h$ denotes the hyperplane class in $\CP$. Hence, a regular fiber is a surface of genus $g$. The exceptional sphere $e_{1}$ in $\CP\#\CPb$, which intersects $f = h - e_{1}$ once positively, lifts to two disjoint $-1$ sphere sections in $Y(g)$. 

\end{proof}

The proof of the following lemma can be extracted from \cite{GS} [Ex 7.3.27, page 268]; we omit proof. 

\begin{lem}\label{E1} The double branched cover $W(g)$ of $\CP$ along a smooth algebraic curve $B$ in the linear system $|2(g+1)\tilde{L}|$ can be decomposed as the fiber sum of two copies of $\CP\#(g+1)^2\CPb$ along a a complex curve of genus equal $g(g-1)/2$. Moreover, $W(g)$ admits a genus $g$ Lefschetz pencil with two base points, and $Y(g) =  W(g)\#2\CPb$. 
\end{lem}

\begin{exmp}\label{Ex} In this example, we study the topology of complex surfaces $W(g)$ in some details. Recall that by Lemma \ref{E1} the complex surface $W(g)$ is the fiber sum of two copies of the rational surface $\CP\#(g^{2}+2g+1)\CPb$ along the complex curve $\Sigma$ of genus $g(g-1)/2$ and self-intersection zero. Using the fiber sum decomposition, we compute the Euler characteristic and the signature of $W(g)$ as follows: $e(W(g)) = 2e(\CP\#(g^{2}+2g+1)\CPb) - 2e(\Sigma) = 4g^2 + 2g + 4$, and $\sigma(W(g)) = 2 \sigma(\CP\#(g^{2}+2g+1)\CPb) = -2(g^{2}+2g)$. Next, we recall from \cite{Fu} that $\CP\#(g^{2}+2g+1)\CPb = {\Phi}_{g(g-1)/2}(1) \cup N_{g(g-1)/2}(1)$, where ${\Phi}_{g(g-1)/2}(1)$ and $N_{g(g-1)/2}(1)$ are Milnor fiber and generalized Gompf nucleus in $\CP\#(g^{2}+2g+1)\CPb$ respectively. Notice that such decomposition shows that the intersection form of $\CP\#(g^{2}+2g+1)\CPb$ splits as $N \oplus M(g)$, where $N = \bigl(\begin{smallmatrix} 0&1\\ 1&-1\end{smallmatrix}\bigr)$ and $M(g)$ is a matrix whose entries are given by a negative definite plumbing tree in the Figure~\ref{fig:plumb}. Consequently, we obtain the following decomposition of the intersection form of $W(g)$:  $2M(g) \oplus H \oplus g(g-1) H$, where $H$ is a hyperbolic pair. Let us choose the following basis which realizes the intesection matrix $M(g) \oplus N$ of $\CP\#(g^{2}+2g+1)\CPb$:  $ <f = (g+1)h - e_1 - \ \cdots  \ - e_{(g+1)^{2}}, \ e_{(g+1)^{2}}, \ e_1 - e_2, \ e_2 - e_3, \ \cdots, \ e_{(g+1)^{2}-2} - e_{(g+1)^{2}-1}, \ h - e_{(g+1)^{2}-(g+1)} - \cdots - e_{(g+1)^{2}-2} - e_{(g+1)^{2}-1}>$. Observe that the last $(g+1)^{2} - 1$ classes can be represented by spheres and their self-intersections are given as in the Figure~\ref{fig:plumb}. $f$ is the class of fiber of the genus $g(g-1)/2$ Lefschetz fibration on $\CP\#(g^{2}+2g+1)\CPb$ and $e_{(g+1)^{2}}$ is a sphere section of self-intersection $-1$. Using the generalized fiber sum decomposition of $W(g)$, it is not hard to see the surfaces that generate the intersection matrix $2M(g) \oplus H \oplus g(g-1) H$. The two copies of the Milnor fiber ${\Phi}_{g(g-1)/2}(1) \subset \CP\#(g^{2}+2g+1)\CPb$ are in $W(g)$, providing $2((g+1)^2 - 1)$ spheres of self-intersections $-2$ and $-g$ (corresponding to the classes $\{ e_1 - e_2, \ e_2 - e_3, \ \cdots, \ e_{(g+1)^{2}-2} - e_{(g+1)^{2}-1}, \ h - e_{(g+1)^{2}-(g+1)} - \cdots - e_{(g+1)^{2}-3} + e_{(g+1)^{2}-2} + e_{(g+1)^{2}-1} \}$ and $\{ {e_1}' - {e_2}', \ {e_2}' - {e_3}', \ \cdots, \ e_{(g+1)^{2}-2}' - e_{(g+1)^{2}-1}', \ h' - e_{(g+1)^{2}-(g+1)} - \cdots - e_{(g+1)^{2}-3}' - e_{(g+1)^{2}-2}' - e_{(g+1)^{2}-1}'\}$), realize two copies of $M(g)$. One copy of hyperbolic pair $H$ comes from an identification of the fibers $f$ and $f'$, and a sphere section $\sigma$ of self-intersection $-2$ obtained by sewing the sphere sections $e_{(g+1)^{2}}$ and ${e_{(g+1)^2}}'$. The remaining $g(g-1)$ copies of $H$ come from $g(g-1)$ rim tori and their dual $-2$ spheres (see related discussion in \cite{GS}, page 73)). These $4g^2 + 2g + 2$ classes generate $H_2$ of $W(g)$. Furthermore, using the formula for the canonical class of the generalized symplectic sum and the adjunction inequality, we compute $K_{W(g)} = (g-2)(h + h')$. Also, the class of the genus $g$ surface of square $2$ of the genus $g$ Lefschetz pencil on $W(g)$ is given by $h + h'$. As a consequence, the class of the genus $g$ fiber in $W(g)\#2\CPb$ is given by $h + h' - E_{1} - E_{2}$, where $E_{1}$ and $E_{2}$ are the homology classes of the exceptional spheres of the blow-ups at the points $p_1$ and $p_2$, the base points of the pencil. We can also verify the symplectic surface $\Sigma$, given by the class $h + h' - E_{1} - E_{2}$, has genus $g$ by applying the adjunction formula to $(W(g)\#2\CPb, \Sigma)$: $g(\Sigma) = 1 + 1/2(K_{W(g)\#2\CPb} \cdot [\Sigma] + [\Sigma]^2) = 1 + ((g-2)(h + h') + E_{1} + E_{2}) \cdot (h + h' - E_{1} - E_{2}) + (h + h' - E_{1} - E_{2})^2)/2 = 1 + (2(g-2) + 2)/2 = g$. We can notice from the intersection form of $W(g)$ that all rim tori can be chosen to have no intersections with the genus $g$ surface in the pencil given by the homology class $h + h'$. Thus, the genus $g$ fiber $\Sigma$ can be chosen to be disjoint from the rim tori that descend to $W(g)\#2\CPb$. 
\end{exmp}

\begin{figure}[ht]
\begin{center}
\includegraphics[scale=.63]{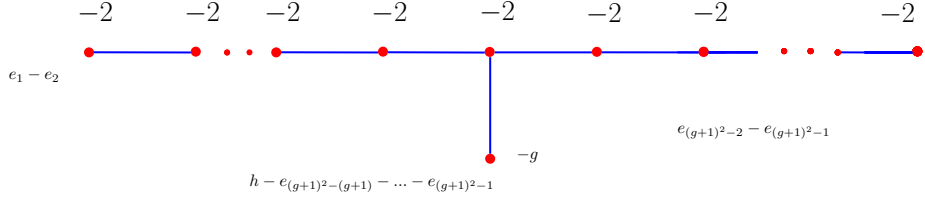}
\caption{Plumbing tree for ${\Phi}_{g(g-1)/2}(1)$}
\label{fig:plumb}
\end{center}
\end{figure}

Let $k$ be any nonnegative integer, and $\mathbb{F}_{k}$ denote $k$-th Hirzebruch surface. Recall that $\mathbb{F}_{k}$ admits the structure of holomorphic $\CPS$ bundle over $\CPS$ with two disjoint holomorphic sections $\Delta_{+k}$ and $\Delta_{-k}$ with $\Delta_{\pm k} = \pm k$.    

\begin{lem}\label{E3} The genus $g$ Lefschetz fibration on $Z(g)$ over $\mathbb{S}^2$ with the monodromy $(c_1c_2 \cdots c_{2g})^{2(2g+1)} = 1$ can be obtained as the $2$-fold cover of $\mathbb{F}_{2}$ branched over the disjoint union of a smooth curve $C$ in the linear system $|(2g+1)\Delta_{+2}|$ and $\Delta_{-2}$  
\end{lem}

\begin{proof} The Lefschetz fibration on $Z(g) \rightarrow  \CPS$ obtained by composing the branched cover map $Z(g) \rightarrow \mathbb{F}_{2}$ with the bundle map $\mathbb{F}_{2} \rightarrow \CPS$. A generic fiber is the double cover of a sphere fiber of $\mathbb{F}_{2}$ branched over $2g + 2$ points. The monodromy of this Lefschetz fibration can be derived from the braid monodromy of the branch curve $C \cup \Delta_{-2}$. The fibration admits a holomorphic sphere section $S$ with $S^2 = -1$, which is obtained by lifting $\Delta_{-2}$ to $Z(g)$.

\end{proof}

\subsection{Rational Blowdown} In this subsection, we review the rational blowdown surgery introduced by Fintushel-Stern \cite{FS1}. For details the reader is referred to \cite{FS1,P1}.

Let $p \geq  2$ and $C_p$ be the smooth $4$-manifold obtained by plumbing disk bundles over the $2$-sphere according to the following linear diagram  

 \begin{picture}(100,60)(-90,-25)
 \put(-12,3){\makebox(200,20)[bl]{$-(p+2)$ \hspace{6pt}
                                  $-2$ \hspace{96pt} $-2$}}
 \put(4,-25){\makebox(200,20)[tl]{$u_{p-1}$ \hspace{25pt}
                                  $u_{p-2}$ \hspace{86pt} $u_{1}$}}
  \multiput(10,0)(40,0){2}{\line(1,0){40}}
  \multiput(10,0)(40,0){2}{\circle*{3}}
  \multiput(100,0)(5,0){4}{\makebox(0,0){$\cdots$}}
  \put(125,0){\line(1,0){40}}
  \put(165,0){\circle*{3}}
\end{picture}

\noindent where each vertex $u_{i}$ of the linear diagram represents a disk bundle over $2$-sphere with the given Euler number. 

The boundary of $C_p$ is the lens space $L(p^2, 1 - p)$ which also bounds a rational ball $B_p$ with $\pi_1(B_p) = {\mathbb{Z}}_p$ and $\pi_1(\partial B_p) \rightarrow  \pi_1(B_p)$ surjective. If $C_p$ is embedded in a $4$-manifold $X$ then the rational blowdown manifold $X_p$ is obtained by replacing $C_p$ with $B_p$, i.e., $X_p = (X \setminus C_p) \cup B_p$. If $X$ and $X \setminus C_p$ are simply connected, then so is $X_p$. The following lemma is easy to check, so we omit the proof.

\begin{lem}\label{thm:rb} $b_{2}^{+}(X_p) = {b_2}^{+}(X)$, $e(X_p) = e(X) - (p-1)$, $\sigma(X_p) = \sigma(X) + (p-1)$, ${c_1}^{2}(X_p) = {c_1}^2(X) + (p-1)$, and $\chi_{h}(X_p) = \chi_{h}(X)$.
\end{lem}

We now collect some theorems on rational blowdown for later use.

\begin{thm}\label{SW1} \cite{FS1, P1}. Suppose $X$ is a smooth 4-manifold with $b_{2}^{+}(X) > 1$ which contains a configuration $C_{p}$. If $L$ is a SW basic class of $X$ satisfying $L\cdot u_{i} = 0$ for any i with $1 \leq i \leq p-2$  and $L\cdot u_{p-1} = \pm p$, then $L$ induces a SW basic class $\bar L$ of $X_{p}$ such that $SW_{X_{p}}(\bar L) = SW_{X}(L)$.  

\end{thm}

\begin{thm}\label{SW2} \cite{FS1, P1} If a simply connected smooth $4$-manifold $X$ contains a configuration $C_{p}$, then the SW-invariants of $X_{p}$ are completely determined by those of $X$. That is, for any characteristic line bundle $\bar{L}$ on $X_{p}$ with $SW_{X_{p}}(\bar{L}) \ne 0$, there exists a characteristic line bundle $L$ on $X$ 
such that $SW_{X}(L) = SW_{X_{p}}(\bar{L})$.

\end{thm}

\begin{thm}[\cite{EN},\cite{EG} $(p=2)$, \cite{EMVHM} $(p\geq 3)$]\label{EN}
Let $\varrho$, $\varrho^{\prime}$ be positive relators of $\Gamma_g$, and let $X_\varrho$, $X_{\varrho^{\prime}}$ be the corresponding Lefschetz fibrations over $\mathbb{S}^2$, respectively. Suppose that $\varrho^\prime$ is obtained by applying a daisy substitution of type $p$ to $\varrho$. Then, $X_{\varrho^\prime}$ is a rational blowdown of $X_\varrho$ along a configuration $C_{p}$. Therefore, we have
\begin{align*}
\sigma(X_\varrho^\prime)=\sigma(X_{\varrho})+(p-1), \ \ \ \mathrm{and} \ \ \ e(X_\varrho^\prime)=e(X_{\varrho})-(p-1). 
\end{align*}
\end{thm}

\subsection{Knot Surgery} In this subsection, we briefly review the knot surgery operation, which gives rise to mutually non-diffeomorphic manifolds. For the details, the reader is referred to \cite{FS2}. 

Let $X$ be a $4$-manifold with ${b_2}^{+}(X) > 1$ and contain a homologically essential torus $T$ of self-intersection $0$. Let $N(K)$ be a tubular neighborhood of $K$ in $\mathbb{S}^3$, and let $T \times D^2$ be a tubular neighborhood of $T$ in $X$. The knot surgery manifold $X_{K}$ is defined by $X_K = (X \setminus (T \times D^2)) \cup (\mathbb{S}^1 \times (\mathbb{S}^3 \setminus N(K))$ where two pieces are glued in a way that the homology class of $[pt \times \partial  D^2]$ is identifed with $[pt \times \lambda]$ where $\lambda$ is the class of the longitude of knot $K$. Fintushel and Stern proved the theorem that shows Seiberg-Witten invariants of $X_{K}$ can be completely determined by the Seiberg-Witten invariant of $X$ and the Alexander polynomial of $K$ \cite{FS2}. Moreover, if $X$ and $X \setminus T$ are simply connected, then so is $X_K$. 

\begin{thm}\label{thm:knotsurgery} Suppose that $\pi_{1}(X) = \pi_{1}(X \setminus T) = 1$ and $T$ lies in a cusp neighborhood in $X$. Then $X_{K}$ is homeomorphic to $X$ and Seiberg-Witten invariants of $X_{K}$ is $SW_{X_K} = SW_{X} \cdot \Delta_{K}(t^2)$, where $t = t_{T}$ (in the notation of \cite{FS2}) and $\Delta_{K}$ is the symmetrized Alexander polynomial of $K$. If the Alexander polynomial $\Delta_{K}(t)$ of knot $K$ is not monic then $X_K$ admits no symplectic structure. Moreover, if $X$ is symplectic and $K$ is a fibered knot, then $X_{K}$ admits a symplectic structure.
\end{thm}

\section{Lemmas}

In this section, we construct some relations by applying elementary transformations. These relations will be used to construct new relations obtained by daisy substitutions in Section~\ref{new words}. 

Let $a_1,\ldots,a_k$ be a sequence of simple closed curves on an oriented surface such that $a_i$ and $a_j$ are disjoint if $|i-j|\geq 2$ and that $a_i$ and $a_{i+1}$ intersect at one point. For simplicity of notation, we write $a_i$, $_f(a_i)$ instead of $t_{a_i}$, $t_{f(a_i)}=ft_{a_i}f^{-1}$, respectively. 
Moreover, write 
\begin{align*}
&b_i = {}_{a_{i+1}}(a_i)& &\mathrm{and}& &\bar{b}_i = {}_{a_{i+1}^{-1}}(a_i).&
\end{align*}
Below we denote the arrangement using the conjugation (i.e. the cyclic permutation) and the arrangement using the relation (i) by $\xrightarrow[]{C}$ and $\xrightarrow[]{(\mathrm{i})}$, respectively.
We recall the following relation:
\begin{align*}
&a_{i+1} \cdot a_i \sim b_i \cdot a_{i+1},& &\mathrm{and}& &a_i \cdot a_{i+1} \sim a_{i+1} \cdot \bar{b}_i.&
\end{align*}
In particular, we note that
\begin{align*}
&a_i \cdot a_j \sim a_j \cdot a_i& &\mathrm{for} \ |i-j|>1.&
\end{align*}

By drawing the curves, it is easy to verify that for $m=1,\ldots,k-1$ and $i=m,\ldots,k-1$, 
\begin{align*}
a_k a_{k-1} \cdots a_{m+1} a_m (a_{i+1})=a_i \ \ \ \mathrm{and} \ \ \ a_m a_{m+1} \cdots a_{k-1} a_k (a_i)=a_{i+1}. 
\end{align*}
Using the relation $t_{f(c)}=ft_cf^{-1}$, we obtain the followings:
\begin{align}
&(a_k a_{k-1} \cdots a_{m+1} a_m) \cdot a_{i+1} \sim a_i \cdot (a_k a_{k-1} \cdots a_{m+1} a_m), \label{1} \\
&(a_m a_{m+1} \cdots a_{k-1} a_k) \cdot a_i \sim a_{i+1} \cdot (a_m a_{m+1} \cdots a_{k-1} a_k). \label{2}
\end{align}

\begin{lem}\label{lem3.3}
For $2\leq k$, we have the following relations:
\begin{align*}
\mathrm{(a)}\hspace{10pt}&(a_{k-1} a_{k-2} \cdots a_2 a_1) \cdot (a_k a_{k-1} \cdots a_2 a_1) \sim a_k^k \cdot \bar{b}_{k-1} \cdots \bar{b}_2 \bar{b}_1, \\[3pt]
\mathrm{(b)}\hspace{10pt}&(a_1 a_2 \cdots a_{k-1} a_k) \cdot (a_1 a_2 \cdots a_{k-2} a_{k-1}) \sim b_1 b_2 \cdots b_{k-1} \cdot a_k^k, \\[3pt]
\end{align*}
\end{lem}

\begin{proof}
The proof will be given by induction on $k$. Suppose that $k=2$. Then, we have
\begin{align*}
\red{a_1} a_2 a_1 \xrightarrow[]{(\ref{1})} a_2 a_1 \red{a_2} \sim a_2^2 \cdot \bar{b}_1. 
\end{align*}
Hence, the conclusion of the Lemma holds for $k=2$.

Let us assume inductively that the relation holds for $k=j$. 
Then, 
\begin{align*}
&(\red{a_j} a_{j-1} \cdots a_1) \cdot (\red{a_{j+1}} a_j \cdots a_1) \\
&\sim \red{a_j a_{j+1}} \cdot (a_{j-1} \cdots a_1) \cdot (a_j \cdots a_1) \\
&\sim a_j a_{j+1} \cdot \blue{a_j^{j+1}} \cdot \bar{b}_{j-1} \cdots  \bar{b}_1 \\
&\xrightarrow[]{(\ref{2})} \blue{a_{j+1}^{j+1}} \cdot a_j a_{j+1} \cdot \bar{b}_{j-1} \cdots  \bar{b}_1 \\
&\sim a_{j+1}^{j+1} \cdot a_{j+1} \cdot \bar{b}_j \cdot \bar{b}_{j-1} \cdots  \bar{b}_1. 
\end{align*}
This proves part (a). The proof of (b) is similar, and therefore omitted.  
 
\end{proof}

\begin{lem}\label{lem3.4}
Let $l\geq 0$. We define an element $\phi$ to be 
\begin{align*}
&\phi_{l}=a_{2l+1}^{l+1} a_{2l-1}^{l} \cdots a_5^3 a_3^2 a_1. 
\end{align*}
Let $D$ and $E$ be two products of right-handed Dehn twists and write them as $D=d_1\cdots d_{k_1}$ and $E=e_1\cdots e_{k_2}$, respectively. 
If a word $W_1$ is obtained by applying a sequence of the conjugation and the elementary transformations to a word $W_2$, then we denote it by $\sim_C$. 
For $l\geq 1$, we have the following:
\begin{align*}
\mathrm{(a)}\hspace{5pt}&D \cdot a_{2l} \cdots a_2 a_1 \cdot a_{2l+1} \cdots a_2 a_1 \cdot E \\
& \sim_C {}_{\phi_l}(D) \cdot (a_{2l+1}^{l+1} \cdot a_{2l} \cdots a_2 a_1) \cdot (b_{2l} \cdots b_4 \cdot b_2) \cdot {}_{\phi_l}(E), \\[3pt]
\mathrm{(b)}\hspace{5pt}&D \cdot a_1 a_2 \cdots a_{2l+1} \cdot a_1 a_2 \cdots a_{2l} \cdot E \\
& \sim_C {}_{\phi_l^{-1}}(D) \cdot (\bar{b}_2 \cdot \bar{b}_4 \cdots \bar{b}_{2l}) \cdot (a_1 a_2 \cdots a_{2l} \cdot a_{2l+1}^{l+1}) \cdot {}_{\phi_l^{-1}}(E), 
\end{align*}
where ${}_{\phi_l}(D)={}_{\phi_l}(d_1)\cdots {}_{\phi_{l}}(d_{k_1})$ and ${}_{\phi_l}(E)={}_{\phi_l}(e_1)\cdots {}_{\phi_{l}}(e_{k_2})$. 
\end{lem}

\begin{proof}
For $1\leq m\leq 2l-1$ and $m\leq i\leq 2l-1$, we have the following equalities from (\ref{1}) and (\ref{2}):
\begin{align}
&(a_{2l} \cdots  a_1 \cdot a_{2l+1} \cdots a_m) \cdot a_{i+2} \sim a_i \cdot (a_{2l} \cdots a_1 \cdot a_{2l+1} \cdots a_m) \label{3}\\
&a_{i+2} \cdot (a_m \cdots a_{2l+1} \cdot a_1 \cdots a_{2l}) \sim (a_m \cdots a_{2l+1} \cdot a_1 \cdots a_{2l}) \cdot a_i. \label{4}
\end{align}

We first show (a). Since 
\begin{align*}
&{}_{\phi_{l-1}}(D) \cdot (a_{2l} a_{2l-1} \cdots a_1) \cdot (b_{2l} b_{2l-2} \cdots b_2) \cdot \red{a_{2l+1}^{l+1}} \cdot {}_{\phi_{l-1}}(E) \\
&\xrightarrow[]{C} {}_{\red{a_{2l+1}^{l+1}}\phi_{l-1}}(D)\cdot \red{a_{2l+1}^{l+1}} \cdot (a_{2l} a_{2l-1} \cdots a_1) \cdot (b_{2l} b_{2l-2} \cdots b_2) \cdot {}_{\red{a_{2l+1}^{l+1}}\phi_{l-1}}(E) \\
&= {}_{\phi_l}(D)\cdot a_{2l+1}^{l+1} \cdot (a_{2l} a_{2l-1} \cdots a_1) \cdot (b_{2l} b_{2l-2} \cdots b_2) \cdot {}_{\phi_l}(E), 
\end{align*}
it is sufficient to prove 
\begin{align*}
&D \cdot (a_{2l} a_{2l-1} \cdots a_1) \cdot (a_{2l+1} a_{2l} \cdots a_1) \cdot E \\
&\sim {}_{\phi_{l-1}}(D) \cdot (a_{2l} a_{2l-1} \cdots a_1) \cdot (b_{2l} b_{2l-2} \cdots b_2)\cdot a_{2l+1}^{l+1} \cdot {}_{\phi_{l-1}}(E). 
\end{align*}
The proof is by induction on $l$. Suppose that $l=0$. Then, we have
\begin{align*}
&D \cdot a_2 a_1 \cdot a_3 a_2 \red{a_1} \cdot E \\
&\xrightarrow[]{C} {}_{\red{a_1}}(D) \cdot \red{a_1} \cdot a_2 a_1 \cdot a_3 a_2 \cdot {}_{\red{a_1}}(E) \\
&\xrightarrow[]{(\ref{3})} {}_{a_1}(D) \cdot a_2 a_1 \cdot a_3 a_2 \cdot \red{a_3} \cdot {}_{a_1}(E) \\
&\sim {}_{a_1}(D) \cdot a_2 a_1 \cdot b_2 \cdot a_3 \cdot a_3 \cdot {}_{a_1}(E). 
\end{align*}
Since $\phi_0=a_1$, the conclusion of the Lemma holds for $l=0$.

Let us assume, inductively, that the relation holds for $l=j$. 
Note that since $\phi_{j-1}(a_{2j+2})=a_{2j+2}$ and $\phi_{j-1}(a_{2j+3})=a_{2j+3}$, we have 
\begin{align*}
{}_{\phi_{j-1}}(D a_{2j+2} a_{2j+1} a_{2j+3} a_{2j+2})={}_{\phi_{j-1}}(D) a_{2j+2} a_{2j+1} a_{2j+3} a_{2j+2}. 
\end{align*}
Since $a_{2j+3}$ is disjoint from $b_2,b_4,\ldots,b_{2j}$ and $\phi_j=a_{2j+1}^{j+1}\phi_{j-1}$, we have
\begin{align*}
&D \cdot (\red{a_{2j+2} a_{2j+1}} a_{2j} \cdots a_1) \cdot (\red{a_{2j+3} a_{2j+2}} a_{2j+1} \cdots a_1) \cdot E \\
&\sim D \cdot \red{a_{2j+2} a_{2j+1} a_{2j+3} a_{2j+2}} \cdot (a_{2j} \cdots a_1) \cdot (a_{2j+1} \cdots a_1) \cdot E \\
&\sim {}_{\phi_{j-1}}(D) \cdot a_{2j+2} a_{2j+1} a_{2j+3} a_{2j+2} \cdot (a_{2j} \cdots a_1) \cdot (b_{2j} \cdots b_2) \cdot a_{2j+1}^{j+1} \cdot {}_{\phi_{j-1}}(E) \\
&\xrightarrow[]{C} {}_{\phi_j}(D) \cdot a_{2j+1}^{j+1} \cdot \red{a_{2j+2} a_{2j+1} a_{2j+3} a_{2j+2}} \cdot (a_{2j} \cdots a_1) \cdot (b_{2j} \cdots b_2) \cdot {}_{\phi_j}(E) \\
&\sim {}_{\phi_j}(D) \cdot a_{2j+1}^{j+1} \cdot (\red{a_{2j+2} a_{2j+1}} a_{2j} \cdots a_1) \cdot \red{a_{2j+3} a_{2j+2}} \cdot (b_{2j} \cdots b_2) \cdot {}_{\phi_j}(E) \\
&\xrightarrow[]{(\ref{3})} {}_{\phi_j}(D) \cdot (a_{2j+2} a_{2j+1} a_{2j} \cdots a_1) \cdot a_{2j+3} a_{2j+2} \cdot a_{2j+3}^{j+1} \cdot (b_{2j} \cdots b_2) \cdot {}_{\phi_j}(E) \\
&\sim {}_{\phi_j}(D) \cdot (a_{2j+2} a_{2j+1} a_{2j} \cdots a_1) \cdot b_{2j+2} \cdot a_{2j+3} \cdot a_{2j+3}^{j+1} \cdot (b_{2j} \cdots b_2) \cdot {}_{\phi_j}(E) \\
&\sim {}_{\phi_j}(D) \cdot (a_{2j+2} a_{2j+1} a_{2j} \cdots a_1) \cdot (b_{2j+2} b_{2j} \cdots b_2)\cdot a_{2j+3}^{j+2} \cdot {}_{\phi_j}(E). 
\end{align*}
This proves part (a) of the lemma. The proof of part (b) is similar and left to the reader.
\end{proof}

\section{A Lift of hyperelliptic relations}
In this section, we construct a relation which gives a lift of a relation, and which is Hurwitz equivalent to $I(g)$, from $\Gamma_g$ to $\Gamma_g^1$. This relation will be used in Section~\ref{new words}.

Suppose $g\geq 2$. 
Let $\Sigma_g^n$ be the surface of genus $g$ with $b$ boundary components $\delta_1,\delta_2,\ldots,\delta_n$. 
Let $\alpha_1,\alpha_2,\ldots,\alpha_{2g},\alpha_{2g+1}, \alpha_{2g+1}^\prime$ and $\delta, \zeta_1,\ldots,\zeta_n$ be the simple closed curves as shown in Figure~\ref{curves}. 
\begin{figure}[hbt]
 \centering
      \includegraphics[width=7cm]{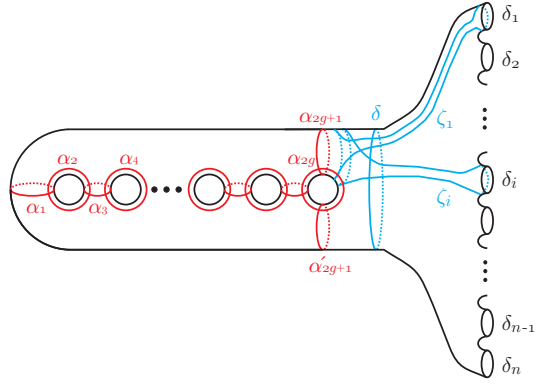}
      \caption{The curves $\alpha_1,\ldots,\alpha_{2g+1},\alpha_{2g+1}^\prime, \delta, \zeta_1,\ldots,\zeta_n$ on $\Sigma_g^n$ and the boundary components $\delta_1,\ldots,\delta_n$ of $\Sigma_g^n$.}
      \label{curves}
 \end{figure}

\begin{lem}\label{lem1}
The following relation holds in $\Gamma_g^n$:
\begin{align*}
(\alpha_1 \alpha_2 \cdots \alpha_{2g})^{2g+1} = (\alpha_1 \alpha_2 \cdots \alpha_{2g-1})^{2g} \cdot \alpha_{2g} \cdots \alpha_2  \alpha_1 \alpha_1 \alpha_2 \cdots \alpha_{2g}. 
\end{align*}
\end{lem}

\begin{proof}
The proof follows from the braid relations, $\alpha_i \cdot \alpha_{i+1} \cdot \alpha_i = \alpha_{i+1} \cdot \alpha_i \cdot \alpha_{i+1}$ and $\alpha_i \cdot \alpha_j= \alpha_j \cdot \alpha_i$ for $|i-j|>1$. 
\end{proof}

\begin{lem}\label{lem2}
The following relation holds in $\Gamma_g^n$:
\begin{align*}
\delta=(\alpha_{2g+1}  \alpha_{2g}  \cdots \alpha_2 \alpha_1 \alpha_1 \alpha_2 \cdots \alpha_{2g} \alpha_{2g+1}^\prime)^2
\end{align*}
In particular, since $\delta_1=\delta$ for $n=1$, this relation is a lift of $H(g)$ from $\Gamma_g$ to $\Gamma_g^1$. 
\end{lem}

\begin{proof}
A regular neighborhood of $\alpha_1\cup \alpha_2 \cup\cdots \cup a_{2g-1}$ is a subsurface of genus $g-1$ with two boundary components, $\alpha_{2g+1}$ and $\alpha_{2g+1}^\prime$. 
Moreover, a regular neighborhood of $\alpha_1\cup \alpha_2 \cup\cdots \cup a_{2g}$ is a subsurface of genus $g$ with one boundary component $\delta$. 
Then, it is well know that the following relations, callled the \textit{chain relations} hold: 
\begin{align*}
&\alpha_{2g+1} \alpha_{2g+1}^\prime = (\alpha_1 \cdots \alpha_{2g-1})^{2g}, 
&\delta = (\alpha_1 \alpha_2 \alpha_3 \cdots \alpha_{2g})^{4g+2}.\\
\end{align*}
By applying the relations above and Lemma~\ref{lem1}, we obtain the following relation. 
\begin{align*}
\delta=(\alpha_{2g+1} \alpha_{2g+1}^\prime \alpha_{2g} \cdots \alpha_2 \alpha_1 \alpha_1 \alpha_2 \cdots \alpha_{2g})^2.
\end{align*}
Since $\alpha_{2g+1}^\prime$ is disjoint from $\alpha_{2g+1}$, by conjugation of $\alpha_{2g+1}^\prime$, we obtain the claim. 
If $n=1$, then it is easily seen that this is a lift of $H(g)$ from $\Gamma_g$ to $\Gamma_g^1$. 
This completes the proof. 
\end{proof}

\begin{lem}\label{lem3}
The following relation holds in $\Gamma_g^n$. 
\begin{align*}
&(\alpha_{2g+1}  \alpha_{2g}  \cdots \alpha_2 \alpha_1 \alpha_1 \alpha_2 \cdots \alpha_{2g} \alpha_{2g+1}^\prime)^2 \\
& \ \sim (\alpha_{2g+1} \alpha_{2g} \cdots \alpha_2 \alpha_1)^2 \cdot (\alpha_1 \alpha_2 \cdots \alpha_{2g} \alpha_{2g+1}) \cdot (\alpha_1 \alpha_2 \cdots \alpha_{2g} \alpha_{2g+1}^\prime). 
\end{align*}
\end{lem}

\begin{proof}
By drawing picture, we find that for each $i=1,2,\ldots,2g$, 
\begin{align*}
&\alpha_{2g+1} \cdots \alpha_2 \alpha_1 \alpha_1 \alpha_2 \cdots \alpha_{2g+1} (\alpha_i) = \alpha_i, \\
&\alpha_{2g} \cdots \alpha_2 \alpha_1 \alpha_1 \alpha_2 \cdots \alpha_{2g} (\alpha_{2g+1}^\prime) = \alpha_{2g+1}. 
\end{align*}
These give the following relations. 
\begin{align}
&\alpha_{2g+1} \cdots \alpha_2 \alpha_1 \alpha_1 \alpha_2 \cdots \alpha_{2g+1} \cdot \alpha_i 
\sim \alpha_i \cdot \alpha_{2g+1} \cdots \alpha_2 \alpha_1 \alpha_1 \alpha_2 \cdots \alpha_{2g+1}, \label{5}\\
&\alpha_{2g} \cdots \alpha_2 \alpha_1 \alpha_1 \alpha_2 \cdots \alpha_{2g} \cdot \alpha_{2g+1}^\prime 
\sim \alpha_{2g+1} \cdot \alpha_{2g} \cdots \alpha_2 \alpha_1 \alpha_1 \alpha_2 \cdots \alpha_{2g}. \label{6}
\end{align}
From these relations, we have 
\begin{align*}
&\alpha_{2g+1}\alpha_{2g} \cdots \alpha_2 \alpha_1 \alpha_1 \alpha_2 \cdots \alpha_{2g} \red{\alpha_{2g+1}^\prime} \alpha_{2g+1} \alpha_{2g} \cdots \alpha_2 \alpha_1 \alpha_1 \alpha_2 \cdots \alpha_{2g} \alpha_{2g+1}^\prime \\
&\xrightarrow[]{(\ref{6})} \alpha_{2g+1} \cdot \red{\alpha_{2g+1}} \alpha_{2g} \cdots \alpha_2 \alpha_1 \alpha_1 \alpha_2 \cdots \alpha_{2g} \alpha_{2g+1} \cdot \blue{\alpha_{2g} \cdots \alpha_2 \alpha_1} \cdot \alpha_1 \alpha_2 \cdots \alpha_{2g} \alpha_{2g+1}^\prime. \\
&\xrightarrow[]{(\ref{5})} (\alpha_{2g+1}\blue{\alpha_{2g} \cdots \alpha_2 \alpha_1}) \cdot (\alpha_{2g+1} \alpha_{2g} \cdots \alpha_2 \alpha_1 \alpha_1 \alpha_2 \cdots \alpha_{2g} \alpha_{2g+1}) \cdot (\alpha_1 \alpha_2 \cdots \alpha_{2g} \alpha_{2g+1}^\prime). 
\end{align*}
This completes the proof. 
\end{proof}

\section{New words in the mapping class group via daisy relation}\label{new words} We define $\phi$ in $\Gamma_g^n$ to be 
\begin{align*}
&\phi=\alpha_{2g+1}^{g+1} \alpha_{2g-1}^{g} \cdots \alpha_5^3 \alpha_3^2 \alpha_1. 
\end{align*}
Note that $\phi(\alpha_{2i-1}) = \alpha_{2i-1}$ for each $i=1,\ldots,g+1$ and $\phi(\alpha_{2g+1}^\prime)=\alpha_{2g+1}^\prime$. 
For simplicity of notation, we write 
\begin{align*}
&\beta_i={}_{\alpha_{i+1}}(\alpha_i),& &\bar{\beta}_i={}_{\alpha_{i+1}^{-1}}(\alpha_i),& &\gamma_{i+1}={}_{\alpha_i}(\alpha_{i+1}),&  &\bar{\gamma}_{i+1}={}_{\alpha_i^{-1}}(\alpha_{i+1}).&
\end{align*}
We denote by $\varphi$, $d_i$, $\bar{d}_i$, $e_{i+1}$ and $\bar{e}_{i+1}$ the images of $\phi$, $\beta_i$, $\bar{\beta}_i$, $\gamma_{i+1}$ and $\bar{\gamma}_{i+1}$ under the map $\Gamma_g^n \to \Gamma_g$, that is, 
\begin{align*}
&\varphi=c_{2g+1}^{g+1} c_{2g-1}^{g} \cdots c_5^3 c_3^2 c_1, 
\end{align*}
and 
\begin{align*}
&d_i={}_{c_{i+1}}(c_i),& &\bar{d}_i={}_{c_{i+1}^{-1}}(c_i),& &e_{i+1}={}_{c_i}(c_{i+1})& &\mathrm{and}& &\bar{e}_{i+1}={}_{c_i^{-1}}(c_{i+1})&
\end{align*}
Then, note that $\varphi(c_{2i-1}) = c_{2i-1}$ for each $i=1,\ldots,g+1$. 
If a word $W_1$ is obtained by applying simultaneous conjugations by $\psi$ to a word $W_2$, then we denote it by $\xrightarrow[]{\psi}$.

Let $x_1,\ldots,x_g$, $x_1^\prime,\ldots,x_g^\prime$ and $y_1,\ldots,y_g$ be the simple closed curves on $\Sigma_g$ given in Figure~\ref{curves3}. 
Moreover, we define $y_{g+1},\ldots,y_{2g-1}$ to be $x_2,\ldots,x_g$, respectively. 
We take the following two daisy relators of type $g-1$ in $\Gamma_g$:
\begin{align*}
&D_{g-1} := c_1^{-1} c_3^{-1} \cdots c_{2g-1}^{-1} \cdot c_{2g+1}^{-(g-2)} \cdot x_1 x_2 \cdots x_g & &\mathrm{and}& \\
&D_{g-1}^\prime := c_1^{-1} c_3^{-1} \cdots c_{2g-1}^{-1} \cdot c_{2g+1}^{-(g-2)} \cdot x_1^\prime x_2^\prime \cdots x_g^\prime,& &&
\end{align*}
and the following daisy relator of type $2(g-1)$:

\begin{align*}
D_{2(g-1)} &:= c_{2g+1}^{-1} c_{2g-1}^{-1} \cdots c_5^{-1} c_3^{-1} \cdot c_3^{-1} c_5^{-1} \cdots c_{2g-1}^{-1} \cdot c_{2g+1}^{-2g+3} \cdot y_1 y_2 \cdots y_{2g-1}\\
&= c_3^{-2} c_5^{-2} \cdots c_{2g-1}^{-2} \cdot c_{2g+1}^{-2g+2} \cdot y_1 y_2 \cdots y_{2g-1}. 
\end{align*}

\begin{figure}[hbt]
 \centering
      \includegraphics[width=7cm]{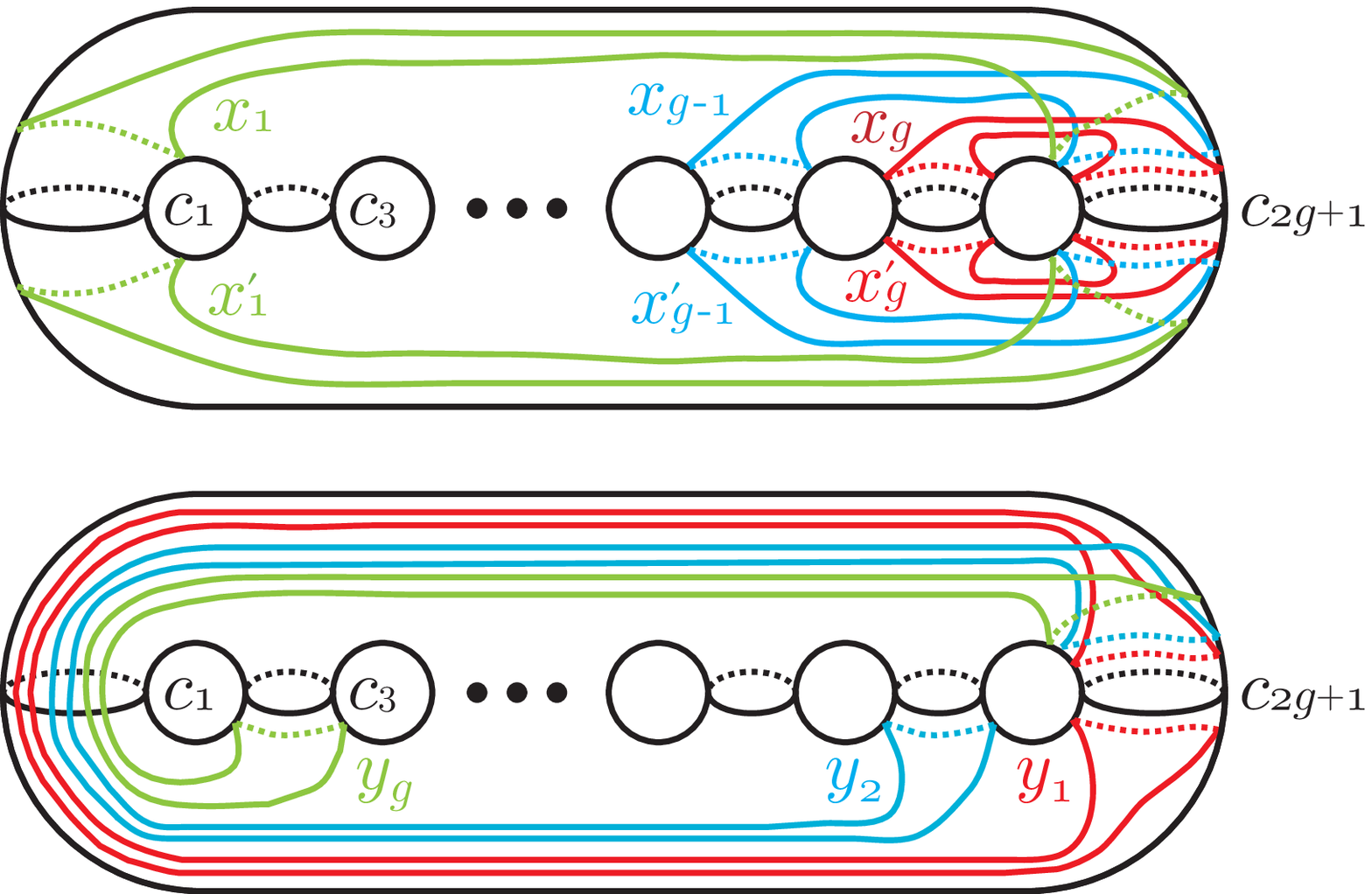}
      \caption{The curves $x_1,\ldots,x_{g}$,$x_1^\prime,\ldots,x_g^\prime$ and $y_1,\ldots,y_g$ on $\Sigma_g$.}
      \label{curves3}
 \end{figure}

Let $\chi_1,\ldots,\chi_g$ be the simple closed curves on $\Sigma_g^n$ given in Figure~\ref{curves2}. 
We denote by $\mathcal{D}_{g-1}$ the following daisy relator of type $g-1$ in $\Gamma_g^n$: 
\begin{align*}
\mathcal{D}_{g-1}= \alpha_1^{-1} \alpha_3^{-1} \cdots \alpha_{2g-1}^{-1} \cdot \alpha_{2g+1}^{-(g-2)} \cdot \chi_1 \chi_2 \cdots \chi_g. 
\end{align*}
Since it is easily seen that $\alpha_i$ and $\chi_i$ are mapped to $c_i$ and $x_i$ under the map $\Gamma_g^n \to \Gamma_g$, we see that the image of this map of $\mathcal{D}_{g-1}$ is $D_{g-1}$. 
\begin{figure}[hbt]
 \centering
      \includegraphics[width=7cm]{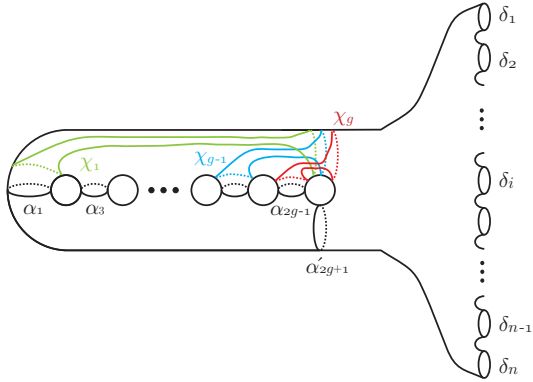}
      \caption{The curves $\chi_1,\ldots,\chi_g$ on $\Sigma_g^n$.}
      \label{curves2}
 \end{figure}

\begin{thm}\label{4.1}There is a positive relator 
\begin{align*}
H(g,1)= {}_{\varphi^{-1}}(\bar{d}_{2g}) \cdots {}_{\varphi^{-1}}(\bar{d}_2) {}_{\varphi^{-1}}(\bar{d}_1) \bar{d}_2 \bar{d}_4 \cdots \bar{d}_{2g} e_2 e_4 \cdots e_{2g} x_1 x_2 \cdots x_g c_{2g+1}^{2g+6}, 
\end{align*}
which is obtained by applying once $D_{g-1}$-substitution to $H(g)$. 
Moreover, the Lefschetz fibration $f_{H(g,1)}$ has $2g+6$ disjoint sections of self-intersection $-1$. 
\end{thm}

\begin{proof}
Note that by Lemma~\ref{lem3.3} (a), we have 
\begin{align}
(\alpha_{2g+1} \alpha_{2g}\cdots \alpha_1)^2 &= \alpha_{2g+1} \cdot (\alpha_{2g} \cdots \alpha_1) \cdot (\alpha_{2g+1} \alpha_{2g} \cdots \alpha_1) \label{7} \\
&\sim \alpha_{2g+1}^{2g+2}\bar{\beta}_{2g}\cdots \bar{\beta}_1 \notag
\end{align}
Moreover, one can check that 
\begin{align}
\alpha_1 \alpha_2 \alpha_3 \cdots \alpha_{2g} \sim (\gamma_2 \gamma_4 \gamma_6 \cdots \gamma_{2g}) \cdot (\alpha_1 \alpha_3 \alpha_5 \cdots \alpha_{2g-1}). \label{8}
\end{align}
Therefore, by Lemma~\ref{lem3.4} (b), we have 
\begin{align*}
&(\alpha_{2g+1} \alpha_{2g} \cdots \alpha_2 \alpha_1)^2 \cdot (\alpha_1 \alpha_2 \cdots \alpha_{2g} \alpha_{2g+1}) \cdot (\alpha_1 \alpha_2 \cdots \alpha_{2g} \alpha_{2g+1}^\prime) \\
&\xrightarrow[]{(\ref{7})} \alpha_{2g+1}^{2g+2} \bar{\beta}_{2g} \cdots \bar{\beta}_1 \cdot (\alpha_1 \alpha_2 \cdots \alpha_{2g} \alpha_{2g+1}) \cdot (\alpha_1 \alpha_2 \cdots \alpha_{2g}) \alpha_{2g+1}^\prime \\
&\sim_C {}_{\phi^{-1}}\{\alpha_{2g+1}^{2g+2} \cdot \bar{\beta}_{2g} \cdots \bar{\beta}_1\} \cdot (\bar{\beta}_2 \bar{\beta}_4 \cdots \bar{\beta}_{2g}) \cdot (\alpha_1 \alpha_2 \cdots \alpha_{2g} \cdot \alpha_{2g+1}^{g+1}) \cdot \alpha_{2g+1}^\prime \\
&= \alpha_{2g+1}^{2g+2} \cdot {}_{\phi^{-1}}(\bar{\beta}_{2g}) \cdots {}_{\phi^{-1}}(\bar{\beta}_1) \cdot (\bar{\beta}_2 \bar{\beta}_4 \cdots \bar{\beta}_{2g}) \cdot (\alpha_1 \alpha_2 \cdots \alpha_{2g} \cdot \alpha_{2g+1}^{g+1}) \cdot \alpha_{2g+1}^\prime \\
&\xrightarrow[]{C} {}_{\phi^{-1}}(\bar{\beta}_{2g}) \cdots {}_{\phi^{-1}}(\bar{\beta}_1) \cdot (\bar{\beta}_2 \bar{\beta}_4 \cdots \bar{\beta}_{2g}) \cdot (\alpha_1 \alpha_2 \cdots \alpha_{2g}) \cdot \alpha_{2g+1}^{3g+3} \cdot \alpha_{2g+1}^\prime \\
&\xrightarrow[]{(\ref{8})} {}_{\phi^{-1}}(\bar{\beta}_{2g}) \cdots {}_{\phi^{-1}}(\bar{\beta}_1) \cdot (\bar{\beta}_2 \bar{\beta}_4 \cdots \bar{\beta}_{2g}) \cdot (\gamma_2 \gamma_4 \cdots \gamma_{2g}) \cdot (\alpha_1 \alpha_3 \cdots \alpha_{2g-1}) \cdot \alpha_{2g+1}^{3g+3} \cdot \alpha_{2g+1}^\prime. 
\end{align*}

Since $\delta$ is a central element of the group generated by $\alpha_1,\ldots,\alpha_{2g+1},\alpha_{2g+1}^\prime$, by Lemma~\ref{lem3}, the operation $\sim_C$ is Hurwitz equivalent (for example, see Lemma 6 in \cite{Auroux2}). 
We have the following relation in $\Gamma_g^{2g+6}$ which is Hurwitz equivalent to the relation $\delta=(\alpha_{2g+1} \alpha_{2g} \cdots \alpha_2 \alpha_1 \alpha_1 \alpha_2 \cdots \alpha_{2g} \alpha_{2g+1}^\prime)^2$: 
\begin{align*}
\delta = {}_{\phi^{-1}}(\bar{\beta}_{2g}) \cdots {}_{\phi^{-1}}(\bar{\beta}_1) \cdot (\bar{\beta}_2 \bar{\beta}_4 \cdots \bar{\beta}_{2g}) \cdot (\gamma_2 \gamma_4 \cdots \gamma_{2g}) \cdot (\alpha_1 \alpha_3 \cdots \alpha_{2g-1}) \cdot \alpha_{2g+1}^{3g+3} \cdot \alpha_{2g+1}^\prime. 
\end{align*}

By applying once $\mathcal{D}_{g-1}$-substitution to this relation, we have the following relation:
\begin{align*}
\delta &= {}_{\phi^{-1}}(\bar{\beta}_{2g}) \cdots {}_{\phi^{-1}}(\bar{\beta}_1) \cdot (\bar{\beta}_2 \bar{\beta}_4 \cdots \bar{\beta}_{2g}) \cdot (\gamma_2 \gamma_4 \cdots \gamma_{2g}) \cdot (\chi_1 \chi_2 \cdots \chi_g) \cdot \alpha_{2g+1}^{2g+5} \cdot \alpha_{2g+1}^\prime, 
\end{align*}
Moreover, by $\alpha_{2g+1}^{2g+5} \cdot \alpha_{2g+1}^\prime \cdot \delta_1 \delta_2 \cdots \delta_{2g+6} = \alpha_{2g+1}^{2g+5} \cdot \delta_1 \delta_2 \cdots \delta_{2g+6} \cdot \alpha_{2g+1}^\prime$ and the daisy-relation $\alpha_{2g+1}^{2g+5} \cdot \delta_1 \delta_2 \cdots \delta_{2g+6} \cdot \alpha_{2g+1}^\prime = \zeta_1 \zeta_2 \cdots \zeta_{2g+6} \cdot \delta$, we obtain 
\begin{align*}
&\delta \cdot \delta_1 \delta_2 \cdots \delta_{2g+6} \\
&= {}_{\phi^{-1}}(\bar{\beta}_{2g}) \cdots {}_{\phi^{-1}}(\bar{\beta}_1) \cdot (\bar{\beta}_2  \bar{\beta}_4 \cdots \bar{\beta}_{2g}) \cdot (\gamma_2 \gamma_4 \cdots \gamma_{2g}) \cdot (\chi_1 \chi_2 \cdots \chi_g) \cdot \\
&\alpha_{2g+1}^{2g+5} \cdot \alpha_{2g+1}^\prime \cdot \delta_1 \delta_2 \cdots \delta_{2g+6} \\ 
&= {}_{\phi^{-1}}(\bar{\beta}_{2g}) \cdots {}_{\phi^{-1}}(\bar{\beta}_1) \cdot (\bar{\beta}_2 \bar{\beta}_4 \cdots \bar{\beta}_{2g}) \cdot (\gamma_2 \gamma_4 \cdots \gamma_{2g}) \cdot (\chi_1 \chi_2 \cdots \chi_g) \cdot \\
&\zeta_1 \zeta_2 \cdots \zeta_{2g+6} \cdot \delta. 
\end{align*}
Therefore, we have the following relation in $\Gamma_g^{2g+6}$:
\begin{align*}
&\delta_1 \cdot \delta_2 \cdots \delta_{2g+6} \\
&={}_{\phi^{-1}}(\bar{\beta}_{2g}) \cdots {}_{\phi^{-1}}(\bar{\beta}_1) \cdot (\bar{\beta}_2 \bar{\beta}_4 \cdots \bar{\beta}_{2g}) \cdot (\gamma_2 \gamma_4 \cdots \gamma_{2g}) \cdot (\chi_1 \chi_2 \cdots \chi_g) \cdot \zeta_1 \zeta_2 \cdots \zeta_{2g+6}
\end{align*}
It is easily seen that $\zeta_1,\ldots,\zeta_{2g+6}$ are mapped to $c_{2g+1}$ under the map $\Gamma_g^{2g+6} \to \Gamma_g$. 
This completes the proof. 
\end{proof}

\begin{thm}\label{4.2}
There is a positive relator 
\begin{align*}
H(g,2)& = {}_{\varphi^{-2}}(\bar{e}_{2g}) \cdots {}_{\varphi^{-2}}(\bar{e}_4)  {}_{\varphi^{-2}}(\bar{e}_2) {}_{\varphi^{-2}}(d_{2g}) \cdots {}_{\varphi^{-2}}(d_4) \cdot {}_{\varphi^{-2}}(d_2) \cdot \\
&\bar{d}_2 \bar{d}_4 \cdots \bar{d}_{2g} e_2 e_4 \cdots e_{2g} (x_1 x_2 \cdots x_g)^2 c_{2g+1}^8. 
\end{align*}
which is obtained by applying twice $D_{g-1}$-substitutions to $H(g)$. Moreover, the Lefschetz fibration $f_{H(g,2)}$ has $8$ disjoint sections of self-intersection $-1$. 
\end{thm}

\begin{proof}
Let $D$ be a product of Dehn twists. Then, we note that by Lemma~\ref{lem3.4} (a) we have 
\begin{align}
&(\alpha_{2g+1} \alpha_{2g} \cdots \alpha_2 \alpha_1)^2 \cdot D = \alpha_{2g+1} \cdot (\alpha_{2g} \cdots \alpha_1) \cdot (\alpha_{2g+1} \alpha_{2g} \cdots \alpha_1) \cdot D \label{9} \\
&\sim_C \alpha_{2g+1}^{g+2} \cdot (\alpha_{2g} \cdots \alpha_2 \alpha_1) \cdot (\beta_{2g} \cdots \beta_4 \beta_2) \cdot {}_{\phi}(D). \notag
\end{align}
Moreover, we have 
\begin{align}
\alpha_{2g}\cdots \alpha_3 \alpha_2 \alpha_1 \sim (\alpha_{2g-1} \cdots \alpha_5 \alpha_3 \alpha_1) \cdot (\bar{\gamma}_{2g} \cdots \bar{\gamma}_6 \bar{\gamma}_4 \bar{\gamma}_2). \label{10}
\end{align}
Then, by Lemma~\ref{lem3.4} (b) we have 
\begin{align*}
&(\alpha_{2g+1} \alpha_{2g} \cdots \alpha_2 \alpha_1)^2 \cdot (\alpha_1 \alpha_2 \cdots \alpha_{2g} \alpha_{2g+1}) \cdot (\alpha_1 \alpha_2 \cdots \alpha_{2g} \alpha_{2g+1}^\prime) \\
&\xrightarrow[]{(\ref{9})} \alpha_{2g+1}^{g+2} \cdot  (\alpha_{2g} \cdots \alpha_2 \alpha_1) \cdot (\beta_{2g} \cdots \beta_4 \beta_2) \cdot {}_{\phi}\{(\alpha_1 \alpha_2 \cdots \alpha_{2g} \alpha_{2g+1}) \cdot (\alpha_1 \alpha_2 \cdots \alpha_{2g} \alpha_{2g+1}^\prime)\} \\
&\xrightarrow[]{(\ref{10})}  \alpha_{2g+1}^{g+2} \cdot (\alpha_{2g-1} \cdots \alpha_3 \alpha_1) \cdot (\bar{\gamma}_{2g} \cdots  \bar{\gamma}_4 \bar{\gamma}_2) \cdot (\beta_{2g} \cdots \beta_4 \beta_2) \cdot \\
&{}_{\phi}\{(\alpha_1 \alpha_2 \cdots \alpha_{2g} \alpha_{2g+1}) \cdot (\alpha_1 \alpha_2 \cdots \alpha_{2g} \alpha_{2g+1}^\prime)\} \\
&\xrightarrow[]{\phi^{-1}} {}_{\phi^{-1}}\{\alpha_{2g+1}^{g+2} \cdot (\alpha_{2g-1} \cdots \alpha_3 \alpha_1) \cdot (\bar{\gamma}_{2g} \cdots \bar{\gamma}_4 \bar{\gamma}_2) \cdot (\beta_{2g} \cdots \beta_4 \beta_2)\} \cdot \\
&(\alpha_1 \alpha_2 \cdots \alpha_{2g} \alpha_{2g+1}) \cdot (\alpha_1 \alpha_2 \cdots \alpha_{2g} \alpha_{2g+1}^\prime) \\
&\sim_C {}_{\phi^{-2}}\{\alpha_{2g+1}^{g+2} \cdot (\alpha_{2g-1} \cdots \alpha_3 \alpha_1) \cdot (\bar{\gamma}_{2g} \cdots \bar{\gamma}_4 \bar{\gamma}_2) \cdot (\beta_{2g} \cdots \beta_4 \beta_2)\} \cdot \\
&(\bar{\beta}_2 \bar{\beta}_4 \cdots \bar{\beta}_{2g}) \cdot (\alpha_1 \alpha_2 \cdots \alpha_{2g} \cdot \alpha_{2g+1}^{g+1}) \cdot \alpha_{2g+1}^\prime \\
&= \alpha_{2g+1}^{g+2} \cdot (\alpha_{2g-1} \cdots \alpha_3 \alpha_1) \cdot {}_{\phi^{-2}}\{(\bar{\gamma}_{2g} \cdots \bar{\gamma}_4 \bar{\gamma}_2) \cdot (\beta_{2g} \cdots \beta_4 \beta_2)\} \cdot \\ 
&(\bar{\beta}_2 \bar{\beta}_4 \cdots \bar{\beta}_{2g}) \cdot (\alpha_1 \alpha_2 \cdots \alpha_{2g} \cdot \alpha_{2g+1}^{g+1}) \cdot \alpha_{2g+1}^\prime \\
&\xrightarrow[]{(\ref{8})} \alpha_{2g+1}^{g+2} \cdot (\alpha_{2g-1} \cdots \alpha_3 \alpha_1) \cdot {}_{\phi^{-2}}\{(\bar{\gamma}_{2g} \cdots \bar{\gamma}_4 \bar{\gamma}_2) \cdot (\beta_{2g} \cdots \beta_4 \beta_2)\} \cdot \\ 
&(\bar{\beta}_2 \bar{\beta}_4 \cdots \bar{\beta}_{2g}) \cdot (\gamma_2 \gamma_4 \cdots \gamma_{2g}) \cdot (\alpha_1 \alpha_3 \cdots \alpha_{2g-1}) \cdot \alpha_{2g+1}^{g+1} \cdot \alpha_{2g+1}^\prime \\
&\xrightarrow[]{C} {}_{\phi^{-2}}\{(\bar{\gamma}_{2g} \cdots \bar{\gamma}_4 \bar{\gamma}_2) \cdot (\beta_{2g} \cdots \beta_4 \beta_2)\} \cdot \\ 
&(\bar{\beta}_2 \bar{\beta}_4 \cdots \bar{\beta}_{2g}) \cdot (\gamma_2 \gamma_4 \cdots \gamma_{2g}) \cdot \alpha_1^2 \alpha_3^2 \cdots \alpha_{2g-1}^2 \cdot \alpha_{2g+1}^{2g+3} \cdot \alpha_{2g+1}^\prime \\
&=({}_{\phi^{-2}}(\bar{\gamma}_{2g}) \cdots {}_{\phi^{-2}}(\bar{\gamma}_4) \cdot \cdot {}_{\phi^{-2}}(\bar{\gamma}_2)) \cdot ({}_{\phi^{-2}}(\beta_{2g}) \cdots {}_{\phi^{-2}}(\beta_4) \cdot \cdot {}_{\phi^{-2}}(\beta_2)) \cdot \\
&(\bar{\beta}_2 \bar{\beta}_4 \cdots \bar{\beta}_{2g}) \cdot (\gamma_2 \gamma_4 \cdots \gamma_{2g}) \cdot \alpha_1^2 \alpha_3^2 \cdots \alpha_{2g-1}^2 \cdot \alpha_{2g+1}^{2g+3} \cdot \alpha_{2g+1}^\prime. 
\end{align*}

Note that by Lemma~\ref{lem3}, the operation $\xrightarrow[]{\phi^{-1}}$ is also Hurwitz equivalent from Lemma 6 in \cite{Auroux2}. 
We have the following relation in $\Gamma_g^8$ which is Hurwitz equivalent to the relation $\delta=(\alpha_{2g+1} \alpha_{2g} \cdots \alpha_2 \alpha_1 \alpha_1 \alpha_2 \cdots \alpha_{2g} \alpha_{2g+1}^\prime)^2$: 
\begin{align}
&\delta = ({}_{\phi^{-2}}(\bar{\gamma}_{2g}) \cdots {}_{\phi^{-2}}(\bar{\gamma}_4) \cdot \cdot {}_{\phi^{-2}}(\bar{\gamma}_2)) \cdot ({}_{\phi^{-2}}(\beta_{2g}) \cdots {}_{\phi^{-2}}(\beta_4) \cdot \cdot {}_{\phi^{-2}}(\beta_2)) \cdot \label{11} \\
&(\bar{\beta}_2 \bar{\beta}_4 \cdots \bar{\beta}_{2g}) \cdot (\gamma_2 \gamma_4 \cdots \gamma_{2g}) \cdot \alpha_1^2 \alpha_3^2 \cdots \alpha_{2g-1}^2 \cdot \alpha_{2g+1}^{2g+3} \cdot \alpha_{2g+1}^\prime. \notag
\end{align}

By applying twice $\mathcal{D}_{g-1}$-substitutions to this relation, we have the following relation:
\begin{align*}
&\delta = ({}_{\phi^{-2}}(\bar{\gamma}_{2g}) \cdots {}_{\phi^{-2}}(\bar{\gamma}_4) \cdot \cdot {}_{\phi^{-2}}(\bar{\gamma}_2)) \cdot ({}_{\phi^{-2}}(\beta_{2g}) \cdots {}_{\phi^{-2}}(\beta_4) \cdot {}_{\phi^{-2}}(\beta_2)) \cdot \\
&(\bar{\beta}_2 \bar{\beta}_4 \cdots \bar{\beta}_{2g}) \cdot (\gamma_2 \gamma_4 \cdots \gamma_{2g}) \cdot (\chi_1 \chi_2 \cdots \chi_g)^2 \cdot \alpha_{2g+1}^7 \cdot \alpha_{2g+1}^\prime. 
\end{align*}
Moreover, by $\alpha_{2g+1}^7 \cdot \alpha_{2g+1}^\prime \cdot \delta_1\delta_2\cdots \delta_8 = \alpha_{2g+1}^7 \cdot \delta_1 \delta_2 \cdots \delta_8 \cdot \alpha_{2g+1}^\prime$ and the daisy-relation $\alpha_{2g+1}^7 \cdot \delta_1 \delta_2 \cdots \delta_8 \cdot \alpha_{2g+1}^\prime = \zeta_1 \zeta_2 \cdots \zeta_8 \cdot \delta$, 
we obtain 
\begin{align*}
\delta \cdot \delta_1\delta_2\cdots \delta_8 &= ({}_{\phi^{-2}}(\bar{\gamma}_{2g}) \cdots {}_{\phi^{-2}}(\bar{\gamma}_4) \cdot \cdot {}_{\phi^{-2}}(\bar{\gamma}_2)) \cdot ({}_{\phi^{-2}}(\beta_{2g}) \cdots {}_{\phi^{-2}}(\beta_4) \cdot \cdot {}_{\phi^{-2}}(\beta_2)) \cdot \\
&(\bar{\beta}_2 \bar{\beta}_4 \cdots \bar{\beta}_{2g}) \cdot (\gamma_2 \gamma_4 \cdots \gamma_{2g}) \cdot (\chi_1 \chi_2 \cdots \chi_g)^2 \cdot \alpha_{2g+1}^7 \cdot \alpha_{2g+1}^\prime \cdot \delta_1\delta_2\cdots \delta_8 \\
&=({}_{\phi^{-2}}(\bar{\gamma}_{2g}) \cdots {}_{\phi^{-2}}(\bar{\gamma}_4) \cdot \cdot {}_{\phi^{-2}}(\bar{\gamma}_2)) \cdot ({}_{\phi^{-2}}(\beta_{2g}) \cdots {}_{\phi^{-2}}(\beta_4) \cdot \cdot {}_{\phi^{-2}}(\beta_2)) \cdot \\
&(\bar{\beta}_2 \bar{\beta}_4 \cdots \bar{\beta}_{2g}) \cdot (\gamma_2 \gamma_4 \cdots \gamma_{2g}) \cdot (\chi_1 \chi_2 \cdots \chi_g)^2 \cdot \zeta_1 \zeta_2 \cdots \zeta_8 \cdot \delta. 
\end{align*}
Therefore, we have the following relation in $\Gamma_g^8$:
\begin{align*}
\delta_1\delta_2\cdots \delta_8 &=({}_{\phi^{-2}}(\bar{\gamma}_{2g}) \cdots {}_{\phi^{-2}}(\bar{\gamma}_4) \cdot \cdot {}_{\phi^{-2}}(\bar{\gamma}_2)) \cdot ({}_{\phi^{-2}}(\beta_{2g}) \cdots {}_{\phi^{-2}}(\beta_4) \cdot \cdot {}_{\phi^{-2}}(\beta_2)) \cdot \\
&(\bar{\beta}_2 \bar{\beta}_4 \cdots \bar{\beta}_{2g}) \cdot (\gamma_2 \gamma_4 \cdots \gamma_{2g}) \cdot (\chi_1 \chi_2 \cdots \chi_g)^2 \cdot \cdot \zeta_1 \zeta_2 \cdots \zeta_8. 
\end{align*}

It is easily seen that $\alpha_i$ and $\chi_i$ are mapped to $c_i$ and $x_i$, respectively, and $\zeta_1,\ldots,\zeta_8$ are mapped to $c_{2g+1}$ under the map $\Gamma_g^8 \to \Gamma_g$. 
This completes the proof. 
\end{proof}


For $j=1,\ldots,2g$, we write 
\begin{align*}
f_j={}_{(c_{j-1}c_{j+1})^{-1}}(\bar{d_j}),
\end{align*}
where $c_0=1$.


\begin{thm}\label{4.7}
Let $g\geq 3$. Then, the monodromy of the hypereliptic Lefschetz fibration given by the word $H(g) = 1$ can be conjugated to contain a daisy relations of type $2(g-1)$. 
\end{thm}
\begin{proof}
Since $\alpha_i$ and $\delta$ are mapped to $c_i$ and $1$ under the map $\Gamma_g^n \to \Gamma_g$, respectively, 
by the equation (\ref{11}) acuired in Theorem~\ref{4.2}, we obtain the following relator: 
\begin{align*}
&1 = ({}_{\varphi^{-2}}(\bar{e}_{2g}) \cdots {}_{\varphi^{-2}}(\bar{e}_4) \cdot \cdot {}_{\varphi^{-2}}(\bar{e}_2)) \cdot ({}_{\varphi^{-2}}(d_{2g}) \cdots {}_{\varphi^{-2}}(d_4) \cdot {}_{\varphi^{-2}}(d_2)) \cdot \\
&(\bar{d}_2 \bar{d}_4 \cdots \bar{d}_{2g}) \cdot (e_2 e_4 \cdots e_{2g}) \cdot c_1^2 c_3^2 \cdots c_{2g-1}^2 \cdot c_{2g+1}^{2g+4}. 
\end{align*}
This relator contains $D_{2(g-1)}$-relator. 
This completes the proof. 
\end{proof}

\begin{thm}\label{4.3} Let $g\geq 3$. Then, the monodromy of the Lefschetz fibration given by the word  $I(g)= 1$ can be conjugated to contain 

\begin{enumerate}[label={\upshape(\roman*)}]
\setlength{\itemsep}{5pt}
\item $g+1$ daisy relations of type $g-1$.
\item $(g+1)/2$ (resp. $g/2$) daisy relations of type $2(g-1)$ for odd (resp. even) $g$.
\end{enumerate}

\end{thm}
\begin{proof}
Let us first prove (i). Since $\varphi(c_{2i-1}) = c_{2i-1}$ for each $i=1,\ldots,g+1$, by Lemma~\ref{lem3.4} (b) we have 
\begin{align}
&(c_1 c_2 \cdots c_{2g} c_{2g+1})^2 = (c_1 c_2 \cdots c_{2g} c_{2g+1}) \cdot (c_1 c_2 \cdots c_{2g}) \cdot c_{2g+1} \label{12} \\
&\sim_C (\bar{d}_2 \bar{d}_4 \cdots \bar{d}_{2g}) \cdot (c_1 c_2 \cdots c_{2g} \cdot c_{2g+1}^{g+1}) \cdot c_{2g+1}. \notag
\end{align}
Moreover,  we have 
\begin{align}
c_1 c_2 c_3 \cdots c_{2g} \sim (e_2 e_4 e_6 \cdots e_{2g}) \cdot (c_1 c_3 c_5 \cdots c_{2g-1}). \label{13}
\end{align}
Therefore, we have 
\begin{align*}
(c_1 c_2 \cdots c_{2g+1})^{2g+2} &\xrightarrow[]{(\ref{12})} (\bar{d}_2 \bar{d}_4 \cdots \bar{d}_{2g} \cdot c_1 c_2 \cdots c_{2g} \cdot c_{2g+1}^{g+2})^{g+1} \\
&\xrightarrow[]{(\ref{13})} (\bar{d}_2 \bar{d}_4 \cdots \bar{d}_{2g} \cdot e_2 e_4 \cdots e_{2g} \cdot c_1 c_3 \cdots c_{2g-1} \cdot c_{2g+1}^{g+2})^{g+1}.
\end{align*}
Then, we see that we can apply once $D_{g-1}^\prime$-substitution and $g$ times $D_{g-1}$-substitutions to $(\bar{d}_2 \bar{d}_4 \cdots \bar{d}_{2g} \cdot e_2 e_4 \cdots e_{2g} \cdot c_1 c_3 \cdots c_{2g-1} c_{2g+1}^{g+2})^{g+1}$. The reason that we apply once $D_{g-1}^\prime$-substitution is to construct a minimal symplectic manifold $Y(g,k)$ in Theorem~\ref{theorem1}(see the proof of Theorem~\ref{theorem1} and Remark 36). This completes the proof of (i).

Next we prove (ii). By Lemma~\ref{lem3.3} (b) we have 
\begin{align}
&(c_1 c_2 \cdots c_{2g} c_{2g+1})^2 = (c_1 c_2 \cdots c_{2g} c_{2g+1}) \cdot (c_1 c_2 \cdots c_{2g}) \cdot c_{2g+1} \label{14} \\
&\sim (d_1 d_2 \cdots d_{2g} \cdot c_{2g+1}^{2g+1}) \cdot c_{2g+1}. \notag
\end{align}
Moreover, we have 
\begin{align}
(c_1 c_2 \cdots c_{2g+1})^2 \sim (c_1^2 c_3^2 \cdots c_{2g+1}^2) \cdot (f_2 f_4 \cdots f_{2g}) \cdot (\bar{d}_2 \bar{d}_4 \cdots \bar{d}_{2g}). \label{15}
\end{align}
From the above relations, we have 
\begin{align*}
&(c_1 c_2 \cdots c_{2g} c_{2g+1})^4 \\
&\xrightarrow[]{(\ref{14})} (d_1 d_2 \cdots d_{2g} \cdot c_{2g+1}^{2g+2}) \cdot (c_1 c_2 \cdots c_{2g} c_{2g+1})^2 \\
&\xrightarrow[]{(\ref{15})} (d_1 d_2 \cdots d_{2g} \cdot c_{2g+1}^{2g+2}) \cdot (c_1^2 c_3^2 \cdots c_{2g+1}^2) \cdot (f_2 f_4 \cdots f_{2g}) \cdot (\bar{d}_2 \bar{d}_4 \cdots \bar{d}_{2g}) \\
&\sim (d_1 d_2 \cdots d_{2g}) \cdot (c_1^2 c_3^2 \cdots c_{2g-1}^2 c_{2g+1}^{2g+4}) \cdot (f_2 f_4 \cdots f_{2g}) \cdot (\bar{d}_2 \bar{d}_4 \cdots \bar{d}_{2g}). 
\end{align*}
From this, we see that $(c_1 c_2 \cdots c_{2g} c_{2g+1})^4$ can be conjugated to contain a daisy relations of type $2(g-1)$. 
Therefore, 
\begin{align*}
&I(g) = 
  \left\{ \begin{array}{ll}
      \displaystyle (c_1 c_2 \cdots c_{2g+1})^{4k} \cdot (c_1 c_2 \cdots c_{2g+1})^2 & \ \ (g=2k) \\
      \displaystyle (c_1 c_2 \cdots c_{2g+1})^{4(k+1)} & \ \ (g=2k+1), 
      \end{array} \right.
\end{align*}
gives the proof of (ii). 

\end{proof}

\begin{thm}\label{4.4} Let $g\geq 3$. Then, the monodromy of the Lefschetz fibration given by the word $G(g) = 1$ can be conjugated to contain $g$ daisy relations of type $2(g-1)$. 
\end{thm}
\begin{proof}
By a similar argument to the proof of Theorem~\ref{4.3}, we have
\begin{align*}
&(c_1 \cdots c_{2g-1} c_{2g})^4 \\
&\sim (d_1 d_2 \cdots d_{2g-1}) \cdot (c_2^2 c_4^2 \cdots c_{2g-2}^2 c_{2g}^{2g+3}) \cdot (f_1 f_3 \cdots f_{2g-1}) \cdot (\bar{d}_1 \bar{d}_3 \cdots \bar{d}_{2g-1}). 
\end{align*}
Let $h:=(a_1\cdots a_{2g} a_{2g+1})$. 
Note that $h(a_i)=a_{i+1}$ for $i=1,\ldots,2g$. 
Then, we have 
\begin{align*}
&(c_1 \cdots c_{2g-1} c_{2g})^4 \\
&\sim (d_1 d_2 \cdots d_{2g-1}) \cdot (c_2^2 c_4^2 \cdots c_{2g-2}^2 c_{2g}^{2g+3}) \cdot (f_1 f_3 \cdots f_{2g-1}) \cdot (\bar{d}_1 \bar{d}_3 \cdots \bar{d}_{2g-1}) \\
&\xrightarrow[]{h} (d_2 d_3 \cdots d_{2g}) \cdot (c_3^2 c_5^2 \cdots c_{2g-1}^2 c_{2g+1}^{2g+3}) \cdot (f_2 f_3 \cdots f_{2g}) \cdot (\bar{d}_2 \bar{d}_4 \cdots \bar{d}_{2g}). 
\end{align*}
From this, we see that $(c_1 c_2 \cdots c_{2g} c_{2g})^4$ can be conjugated to contain a daisy relations of type $2(g-1)$. 
$G(g)=(c_1 \cdots c_{2g-1} c_{2g})^{4g}\cdot(c_1 \cdots c_{2g-1} c_{2g})^2$ gives the proof. 
\end{proof}

\subsection{Non-hyperellipticity of our Lefschetz fibrations}

The purpose of this subsection is to prove that all the Lefschetz fibrations obtained in this paper via daisy substitutions are non-hyperelliptic. The proof will be obtained as a corollary of more general theorem given below  

\begin{thm}\label{nonhyperelliptic} Let $g\geq 3$. Let $f_{\varrho_1}:X_{\varrho_1} \rightarrow \mathbb{S}^2$ be a genus-$g$ hyperelliptic Lefschetz fibration with only non-separating vanishing cycyles, and let $\varrho_1$ 
be a positive relator corresponding to $f$. Let $k_1$ and $k_2$ be non-negative integers such that $k_1+k_2>0$, and let $k$ be a positive integer. 
Then we have the followings: 
\begin{enumerate}[label={\upshape(\roman*)}]
\setlength{\itemsep}{5pt}
\item If we can obtain a positive relator, denoted by $\varrho_2$, by applying $k_1$ times $D_{g-1}$-substitutions and $k_2$ times $D_{g-1}^\prime$-substitutions to $\varrho$, 
then the genus-$g$ Lefschetz fibration $f_{\varrho_2}:X_{\varrho_2} \rightarrow \mathbb{S}^2$ is non-hyperelliptic. 
\item If we can obtain a positive relator, denoted by $\varrho_3$, by applying $k$ times $D_{2(g-1)}$-substitutions to $\varrho$, 
then the genus-$g$ Lefschetz fibration $f_{\varrho_3}:X_{\varrho_3} \rightarrow \mathbb{S}^2$ is non-hyperelliptic. 
\end{enumerate}
\end{thm}

\begin{cor}
All our Lefschetz fibrations are non-hyperelliptic. 
\end{cor}

\begin{proof}[Proof of Theorem~\ref{nonhyperelliptic}]
Let $s_0$ be the number of non-separating vanishing cycles of $f_{\varrho_1}$. 
Note that by Theorem~\ref{sign}, we have $\sigma(X_{\varrho_1})=-\dfrac{g+1}{2g+1}s_0$. 

First, we assume that $f_{\varrho_2}$ is a hyperelliptic Lefschetz fibration. The relators $D_{g-1}$ and $D_{g-1}^\prime$ consist of only Dehn twists about non-separating simple closed curves $c_1,c_3,\ldots,c_{2g+1}$ and $x_1,x_1^\prime,\ldots,x_g,x_g^\prime$ as in Figure~\ref{curves3}. Therefore, we see that $\varrho_2$ consits of only right-handed Dehn twists about non-separating simple closed curves, so $f_{\varrho_2}$ has only non-separating vanishing cycles. 
In particular, the number of non-separating vanishing cycles of $f_{\varrho_2}$ is $s_0-\{(g-1)-1\}(k_1+k_2)$. By Theorem~\ref{sign}, we have 
\begin{align*}
\sigma(X_{\varrho_2})=-\frac{g+1}{2g+1}\{s_0-(g-2)(k_1+k_2)\}.
\end{align*}
On the other hand, since the relators $D_{g-1}$ and $D_{g-1}^\prime$ are daisy relators of type $g-1$, 
by Theorem~\ref{EN}, we have 
\begin{align*}
&\sigma(X_{\varrho_2}) = \sigma(X_{\varrho_1})+(g-2)(k_1+k_2) = -\frac{g+1}{2g+1}s_0+(g-2)(k_1+k_2).& 
\end{align*}
We get a contradiction since the above equality does not hold for $g \geq 3$ and $k_1+k_2 > 0$.  

Next, we assume that $f_{\varrho_3}$ is a hyperelliptic Lefschetz fibration. The relator $D_{2(g-1)}$ consits of Dehn twsts about non-separating simple closed curves $c_3,c_5,\ldots,c_{2g+1}$ and $y_2,y_3,\ldots,y_{2g-1}$ in Figure~\ref{curves3} and a Dehn twist about separating simple closed curve $y_1$ in Figure~\ref{curves3}. Note that $y_{g+1}=x_2, y_{g+2}=x_3, \ldots, y_{2g-1}=x_g$ and that $y_1$ separates $\Sigma_g$ into two surface, one of which has genus $1$. Therefore, $f_{\varrho_3}$ has $s_0-k\{2(g-1)\}$ non-separating vanishing cycles and $k$ separating vanishing cycles which are $y_1$. 
By Theorem~\ref{sign}, we have
\begin{align*}
\sigma(X_{\varrho_3})&=-\frac{g+1}{2g+1}\{s_0-2k(g-1)\}+\left(\frac{4(g-1)}{2g+1}-1\right)k \\ 
&=-\frac{g+1}{2g+1}s_0+\frac{2g^2+2g-7}{2g+1}k
\end{align*}
On the other hand, since the relator $D_{2(g-1)}$ is a daisy relator of type $2(g-1)$, 
by Theorem~\ref{EN}, we have 
\begin{align*}
\sigma(X_{\varrho_3}) = \sigma(X_{\varrho_1})+(2g-3)k = -\frac{g+1}{2g+1}s_0+(2g-3)k. 
\end{align*}
Since $g\geq 3$, we have 
\begin{align*}
&\left(-\frac{g+1}{2g+1}s_0+(2g-3)k\right)-\left(-\frac{g+1}{2g+1}s_0+\frac{2g^2+2g-7}{2g+1}k\right) \\
&=\frac{2(g-1)(g-2)}{2g+1}k>0. 
\end{align*}
This is a contradiction to the above equality.
\end{proof}

\section{Constructing exotic 4-manifolds}

The purpose of this section is to show that the symplectic $4$-manifolds obtained in Theorem \ref{4.3}, part (i), are irreducible. Moreover, by performing the knot surgery operation along a homologically essential torus on these symplectic $4$-manifolds, we obtain infinite families of mutually nondiffeomorphic irreducible smooth structures. 

\begin{thm}\label{theorem1} Let $g \geq 3$ and $M$ be one of the following $4$-manifolds $(g^2 - g + 1)\CP\# (3g^{2} + 3g + 3 -(g-2)k)\CPb$ for $k = 2, \cdots, g+1$. There exists an irreducible symplectic\/ $4$-manifold $Y(g,k)$ homeomorphic but not diffeomorphic to\/ $M$ that can be obtained from the genus $g$ Lefschetz fibration on $Y(g)$ over $\mathbb{S}^2$ with the monodromy $(c_1c_2 \cdots c_{2g+1})^{2g+2} = 1$ in the mapping class group $\Gamma_g$ by applying $k$ daisy substitutions of type $g-1$. 
\end{thm}

\begin{proof} Let $Y(g,k)$ denote the symplectic $4$-manifold obtained from $Y(g) = W(g)\#2\,\CPb$ by applying $k-1$ $D_{g-1}$-subsitutions and one $D_{g-1}^\prime$-substitution as in Theorem~\ref{4.3}. Applying Lemma~\ref{thm:rb}, we compute the topological invariants of $Y(g,k)$ 
 
\begin{eqnarray*}
e(Y(g,k)) &=& e(W(g)\#2\CPb) - k(g-2) =  2(2g^2 + g + 3) - k(g-2),\\
\sigma(Y(g,k)) &=& \sigma (W(g)\#2\CPb) + k(g-2) =  -2(g+1)^2 + k(g-2).
\end{eqnarray*}

Using the factorization of the global monodromy in terms of right-handed Dehn twists of the genus $g$ Lefschetz on $Y(g,k)$ (see Theorem~\ref{4.3}, part (i)), it is easy to check that $\pi_1(Y(g,k))=1$. 

Next, we show that $Y(g,k)$ is non-spin $4$-manifold. Let $c_{1}, \cdots, c_{2g+1}$ be the curves in Figure~\ref{fig:hyper}, and $x_{1}, \cdots, x_{g}, x'_{1}, \cdots, x'_{g}$ be the curves in Figure~\ref{curves3}. Note that  $\bar{d}_i = {}_{c_{i+1}^{-1}}(c_i)$ and  $e_{i+1} = {}_{c_{i}}(c_{i+1})$. The vanishing cycles of Lefschetz fibrations in Theorem~\ref{4.3} includes the curves  $\bar{d}_{2i}$, $e_{2i}$ for $i = 1, \cdots, g$, and $x_j$, $x'_{j}$ for $j = 1, \cdots, g$, and $c_{2g+1}$. In $H_{1}(\Sigma_{g}; \mathbb{Z}_{2})$, we find that $\bar{d}_{2g} =c_{2g+1} + c_{2g}$, $e_{2g} = c_{2g-1} + c_{2g}$ and  $x_{g} = c_{2g-1} + c_{2g+1}$.Therefore, we have $\bar{d}_{2g} + e_{2g} = x_{g}$. In the notation of Theorem~\ref{Spin1}, we have $l=2$ and $\bar{d}_{2g} \cdot e_{2g} = 0$. Therefore, $2+ \bar{d}_{2g} \cdot e_{2g} \equiv 0 (\textrm{mod}\ 2)$.

By Theorem~\ref{Spin2}, $Y(g,k)$ is non-spin, and thus have an odd intersection form. By Freedman's theorem (cf.\ \cite{freedman}), we see that  $Y(g,k)$ is homeomorphic to $(g^2 - g + 1)\CP\# (3g^{2} + 3g + 3 -(g-2)k)\CPb$.

Next, using the fact that $W(g)$ is a minimal complex surface of general type with ${b_{2}}^{+} > 1$ and the blow up formula for the Seiberg-Witten function \cite{FS2}, we compute $SW_{W(g)\#2\,\CPb}$ $= SW_{W(g)} \cdot \prod_{j=1}^{2}(e^{E_{i}} + e^{-E_{i}}) = (e^{K_{W(g)}} + e^{-K_{W(g)}})(e^{E_{1}} + e^{-E_{1}})(e^{E_{2}} + e^{-E_{2}})$, where $E_{i}$ denote the exceptional class of the $i^{th}$ blow-up. By the above formula, the SW basic classes of $W(g)\#2\,\CPb$ are given by $\pm K_{W(g)} \pm E_{1} \pm E_{2}$, and the values of the Seiberg-Witten invariants on these classes are $\pm 1$. Notice that by the Corollary 8.6 in~\cite{FS1}, $Y(g,k)$ has Seiberg-Witten simple type. Furthermore, by applying Theorem~\ref{SW1} and Theorem~\ref{SW2}, we completely determine the Seiberg-Witten invariants of $Y(g,k)$ using the basic classes and invariants of $W(g)\#2\,\CPb$: Up to sign the symplectic manifold $Y(g,k)$ has only one basic classes which descends from the $\pm$ canonical class of $Y(g)$ (see a detailed explanation below). By Theorem ~\ref{SW2}, or Taubes theorem \cite{taubes} the value of the Seiberg-Witten function on these basic classes, $\pm K_{Y(g,k)}$, are ${ \pm 1}$. 

In what follows, we spell out the details of the above discussion. By Theorem~\ref{SW1} and Theorem~\ref{SW2}, we can determine the Seiberg-Witten invariants of $Y(g,k)$ by computing the algebraic intersection number of the basic classes $\pm  K_{W(g)} \pm E_{1} \pm E_{2}$ of $W(g)\#2\,\CPb$, with the classes of spheres of $k$ disjoint $C_{g-1}$ configurations in $Y(g)$. Notice that the leading spheres of the configurations $C_{g-1}$ are the components of the singular fibers of $Y(g)$. By looking the regions on the genus $g$ surface, where the rational blowdowns along $C_{g-1}$ are performed, and the location of the base points of the genus $g$ pencil, we compute the algebraic intersection numbers as follows: Let $S_1^j$ denote the homology class of $-(g+1)$ sphere of the $j$-th configurations $C_{g-1}$ and $S_2^j, \cdots, S_{g-2}^j$ are the homology classes of $-2$ spheres of $C_{g-1}$ in $W(g)\#2\,\CPb$, where $j = 1, 2$. These two rational blowdowns along $C_{g-1}$ chosen such that they correspond to one $D_{g-1}$-subsitutions and one $D_{g-1}^\prime$-substitution as in Theorem~\ref{4.3}, part (i).
 
We have $S_1^1 \cdot E_{1} = 1$, $S_1^1 \cdot E_{2} = 0$, $S_1^2 \cdot E_{2} = 1$, $S_1^2 \cdot E_{1} = 0$, $S_1^j \cdot K_{W{g}} = g-2$, and the canonical divisor does not intersect with $S_i^j$ for $2 \leq i \leq g-1$. Consequently, $S_1^j \cdot \pm (K_{W{g}} + E_{1} + E_{2}) = \pm (g-1)$ for  $j = 1, 2$, and $S_1^j \cdot (\pm K_{W{g}} \pm E_{1} \mp E_{2}) \neq \pm (g-1)$ for one $j$. Observe that among the eight basic classes $\pm K_{W{g}} \pm E_{1} \pm E_{2}$, only $K_{W{g}} + E_{1} + E_{2}$ and $-(K_{W{g}} + E_{1} + E_{2})$ have algebraic intersection $\pm (g-1)$ with $-(g+1)$ spheres of $C_{g-1}$. Thus, Theorem~\ref{SW1} implies that these are only two basic classes that descend to $Y(g,2)$, and consequently to $Y(g,k)$ from $W(g)\#2\,\CPb$.

By invoking the connected sum theorem for Seiberg-Witten invariants, we see that $SW$ function is trivial for $(g^2 - g + 1)\CP\# (3g^{2} + 3g + 3 -(g-2)k)\CPb$. Since the Seiberg-Witten invariants are diffeomorphism invariants, $Y(g,k)$ is not diffeomorphic to $(g^2 - g + 1)\CP\# (3g^{2} + 3g + 3 -(g-2)k)\CPb$. 

The minimality of $Y(g,k)$ is a consequence of the fact that $Y(g,k)$ has no two Seiberg-Witten basic classes $K$ and $K'$ such that $(K - K')^2 = -4$. Notice that $\pm K_{Y(g,k)}$ are only basic classes of $Y(g,k)$, and $(K_{Y(g,k)}^2 - (-K_{Y(g,k)}))^2$ = $4(K_{Y(g,k)}^{2}) \geq 0$. Thus, we conlude that $Y(g,k)$ is symplectically minimal. Furthermore, since symplectic minimality implies irreducibility for simply-connected 4-manifolds \cite{HK}, we deduce that $Y(g,k)$ is also smoothly irreducible.

\end{proof}

The analogus theorem for $g = 2$, using lantern substitution, was proved in \cite{AP}.

\begin{thm}\label{theorem3} There exist an infinite family of irreducible symplectic\/ and an infinite family of irreducible non-symplectic\/ pairwise non-diffeomorphic $4$-manifolds all homeomorphic to $Y(g,k)$.
\end{thm}

\begin{proof} $Y(g,k)$\/ contains $g(g-1)$ Lagrangian tori which are disjoint from the singular fibers of genus $g$ Lefschetz fibration on $Y(g,k)$\/. These tori descend from $W(g)$ (See Example~\ref{Ex}), and survive in $Y(g,k)$ after the rational blowdowns along $C_{g-1}$. These tori are Lagrangian, but we can perturb the symplectic form on $Y(g,k)$ so that one of these tori, say $T$ becomes symplectic. Moreover, $\pi_1(Y(g,k) \setminus T) = 1$, which follows from the Van Kampen's Theorem using the facts that $\pi_1(Y(g,k)) = 1$ and any rim torus has a dual $-2$ sphere (see Proposition 1.2 in \cite{Halic}, or Gompf \cite{gompf}, page 564). Hence, we have a symplectic torus $T$ in $Y(g,k)$\/ of self-intersection $0$ such that $\pi_1(Y(g,k) \setminus T) = 1$. By performing a knot surgery on $T$, inside $Y(g,k)$, we obtain an irreducible $4$-manifold ${Y(g,k)}_K$ that is homeomorphic to $Y(g,k)$. By varying our choice of the knot $K$, we can realize infinitely many pairwise non-diffeomorphic $4$-manifolds, either symplectic or nonsymplectic.\end{proof}

\begin{rmk} We obtain the analogous results as in Theorem~\ref{theorem1} for the genus $g$ Lefschetz fibrations obtained in Theorem~\ref{4.4}. The total space $Z(g)$ of the Lefschetz fibration with the monodromy $(c_1c_2 \cdots c_{2g+1})^{2(2g+1)} = 1$ in the mapping class group $\Gamma_g$ are complex surface of general type with ${b_{2}}^{+} > 1$  and the single blow-up of a minimal complex surface (see \cite{GS}, Section 8.4, p.320-22). The computation of the Seiberg-Witten invariants follows the same lines of argument as that of Theorems \ref{theorem1}. In fact, the SW computation is simpler, since $Z(g)$ only admits one pair of basic classes. Also, the results of our paper can be easily extendable to the Lefschetz fibrations with monodromies $(c_1c_2 \cdots c_{2g-1}c_{2g}{c_{2g+1}}^2c_{2g}c_{2g-1} \cdots c_2c_1)^{2n} = 1$, $(c_1c_2 \cdots c_{2g}c_{2g+1})^{(2g+2)n} = 1$, $(c_1c_2 \cdots c_{2g-1}c_{2g})^{2(2g+1)n} = 1$ for $n \geq 2$. Since the computations are lengthy, we will not present it here. 
\end{rmk}

\begin{rmk} Note that all the Lefschetz fibrations constructed in our paper are non-spin. The fibrations obtained in Theorems~\ref{4.1} and \ref{4.2} admit a section of self-intersection $-1$. The fibrations in Theorem~\ref{4.7} and \ref{4.4} contain separating vanishing cycles (the curve $y_1$ in Figure \ref{curves3} is separating). Thus, the total spaces are non-spin. In general, if we apply the daisy substitution of type $2(g-1)$ to a positive relator in $\Gamma_g$, then the resulting Lefschetz fibration always contains separating vanishing cycles. The fibrations in Theorem~\ref{4.3} do not contain any separating vanishing cycles, but they are non-spin due to Stipsicz's criteria (see proof of Theorem~\ref{theorem1}). 
\end{rmk}

\begin{rmk} It would be interesting to know if the analogue of Theorem~\ref{theorem1} holds for Lefschetz fibrations of Theorem~\ref{4.2}. Their corresponding monodromies are obtained by applying twice $D_{g-1}$-substitutions to $H(g)$. In the opposite direction, we can prove that the total spaces of the Lefschetz fibrations of Theorem~\ref{4.1}, whose monodromies obtained by applying one $D_{g-1}$-substitutions to $H(g)$, are blow-ups of the complex projective plane. Notice that by Theorem~\ref{4.1} they admit at least $2g+6$ sphere sections of self-intersection $-1$. By using the result of Y. Sato \cite{Sato} (see Theorem 1.2, page 194), we see that the total spaces of these Lefschetz fibrations are diffemorphic to $\CP\#(3g+5)\CPb$. 
\end{rmk}

\begin{rmk}
In the proof of Theorem~\ref{4.3}, part (i), if we apply $k$ $D_{g-1}$-substitutions for $k=1,\ldots, g+1$ without applying an $D_{g-1}^\prime$-substitution, then the Lefschetz fibrations over $\mathbb{S}^2$ given by the resulting relations admits a section of self-intersection $-1$ (i.e. the total spaces of the fibrations are non-minimal). 
\end{rmk}

\section*{Acknowledgments} The authors are grateful to the referee for valuable comments and suggestions. A. A. was partially supported by NSF grants FRG-1065955, DMS-1005741 and Sloan Fellowship. N. M. was partially supported by Grant-in-Aid for Young Scientists (B) (No. 13276356), Japan Society for the Promotion of Science.

\end{document}